\documentclass[a4paper, leqno, 12pt]{amsart}
 \usepackage[usenames]{color}
 \usepackage{amssymb, amsthm, amsmath, amscd}
\usepackage{ascmac}
 \usepackage{fancybox}
 \usepackage{ulem}
 \usepackage[all]{xy}
 \usepackage{comment}
 \usepackage{here}
 \input amssym.def
 \input amssym.tex
 \newcommand{\Z}{\mathbb{Z}}

 \newcommand{\Red}[1]{\textcolor{black}{#1}}
 \newcommand{\Blue}[1]{\textcolor{black}{#1}}
  \newcommand{\Ha}[1]{\widehat{#1}}
  \newcommand{\Ov}[1]{\overline{#1}}
 \newcommand{\Ti}[1]{\widetilde{#1}}
  \newcommand{\Ch}[1]{\check{#1}}
  \newcommand{\Bb}[1]{\mathbb{#1}}
  \newcommand{\Fr}[1]{\mathfrak{#1}}
   \newcommand{\Rm}[1]{\mathrm{#1}}
   \newcommand{\Ca}[1]{\mathcal{#1}} 
\newtheorem{prop}{Proposition}[section]
 \newtheorem{thm}[prop]{Theorem}
 \newtheorem{lem}[prop]{Lemma}
 \newtheorem{cor}[prop]{Corollary}

\theoremstyle{remark}
\newtheorem{define}[prop]{Definition}
\newtheorem{rmk}[prop]{Remark}

\numberwithin{equation}{section} 
 
 \makeatletter
\@namedef{subjclassname@2020}{2020 Mathematical Subject Classification}
\makeatother
\begin{document}
\baselineskip=18pt
 \title{Bounds for the order of automorphism groups of cyclic covering fibrations of an elliptic surface}
\author{Hiroto Akaike}

 
 
\maketitle

\begin{abstract}
We study automorphism groups of fibered surfaces for finite cyclic covering fibrations of an elliptic surface.
We estimate the order of a finite subgroup of automorphism groups in terms of the genus of the base curve, the covering degree and the square of the relative canonical divisor.
\end{abstract}

\section*{Introduction}
\Red{
We shall work over the complex number field $\mathbb{C}$.
This is a continuation of our previous work \cite{Aka} studying the order of 
automorphism groups of fibred algebraic surfaces.}

The study of automorphism groups of projective varieties has a long history.
\Red{
Hurwitz's classical theorem states that if $C$ is a smooth projective curve of genus $g(C) \geq 2$,
then} \Blue{the order of the automorphism group} \Red{is bounded by $84(g(C)-1)$.}
Xiao attempted to generalize \Red{this result} to minimal surfaces of general type and gave the upper bound of the order of \Blue{the automorphism groups} (\cite{Xiao1},\cite{Xiao2}).
\Red{While studies of the automorphism groups of surfaces exist in absolute settings} \Blue{as above},
\Red{we can consider a similar problem in relative settings.}
For this purpose, 
we define basic notions of fibered surfaces and its automorphisms.

Let $f:S \to B$ be a surjective morphism from a complex smooth projective surface $S$ 
to a smooth projective curve $B$ with connected fibers. 
We call it a {\it fibration} or \Red{a {\it fibered surface}} of genus $g$ when a general fiber is a curve of genus $g$.
A fibration is called {\it relatively minimal}, when any $(-1)$-curve is not contained in fibers.
Here we call a smooth rational curve $C$ with $C^2=-n$ a $(-n)$-curve.
A fibration is called {\it smooth} when all fibers are smooth, {\it isotrivial} when all of the 
smooth fibers are isomorphic, {\it locally trivial} when it is smooth and isotrivial.
\Red{
Assume that $f: S \to B$ is a relatively minimal fibration of genus $g$.
We denote by $K_{f}=K_{S}-f^{\ast}K_{B}$ a relative canonical divisor.
}

An automorphism of the fibration $f:S \to B$ is a pair of automorphisms $(\kappa_{S},\kappa_{B})\in \Rm{Aut}(S)\times  \Rm{Aut}(B)$
satisfying $ f \circ \kappa_{S} =  \kappa_{B} \circ  f $, that is, the diagram 
\begin{equation*}
\xymatrix{
S\ar[r]^-{\kappa_{S}}\ar[d]_-{f}\ar@{}[rd]|{\circlearrowright}&S\ar[d]^-{f}\\
B\ar[r]_-{\kappa_{B}}&B
}
\end{equation*}
is commutative.
We denote by $\Rm{Aut} (f)$ the group of all automorphisms of $f$.

Our main objects are primitive cyclic covering fibrations:
\Red{
\begin{define}[\cite{Eno}]
Let $f:S \to B$ be a relatively minimal fibration of genus $g\geq2$.
We call it a {\it primitive cyclic covering fibration} of type $(g,1,n)$, when 
there are  a $($not necessarily relatively minimal $)$ fibration $\tilde{\varphi}:\Ti{W} \to B$ of 
genus $1$ ({\it i.e.} elliptic surface) and a classical $n$-cyclic covering
$$
\tilde{\theta}:\Ti{S}=
\mathrm{Spec}_{\Ti{W}}\left(\bigoplus_{j=0}^{n-1} \mathcal{O}_{\Ti{W}}(-j\Ti{\Fr{d}})\right)\to\Ti{W}
$$
branched over a smooth (not necessarily irreducible) curve $\Ti{R} \in |n\Ti{\Fr{d}}|$ 
and $\Ti{\Fr{d}} \in \Rm{Pic}(\Ti{W})$ such that $f$ is the relatively minimal model of 
$\tilde{f}=\tilde{\varphi}\circ\tilde{\theta}$.
\end{define}
}
\noindent
\Red{
Though it seems restrictive, it should be noticed that this class 
contains some classically important fibrations as subclasses.
In fact, one sees immediately that any hyperelliptic (resp. bielliptic) 
fibrations of genus $g$ are necessarily
primitive cyclic covering fibrations of type $(g,0,2)$ (resp. of type 
$(g,1,2)$ when $g\geq 6$).
Here, a bielliptic curve is a smooth projective curve admitting a double covering of an elliptic curve
and a fibration is called {\it bielliptic} if a general fiber is a bielliptic curve.
}
Bielliptic fibrations are \Blue{an interesting class} of fibrations.
For instance, 
it seemed that the lower bound of slopes of fibrations   
increases monotonically with respect to \Red{its gonality}.
But it is not true.
\Red{There exists a tetragonal fibration and its slope is less than the lower bound 
of the slope of trigonal fibrations.
Such tetragonal fibrations are nothing more than bielliptic fibrations (\cite{Bar},\cite{Kon})}.

Arakawa  \cite{Ara} and later Chen \cite{Chen} studied \Red{the automorphism groups} for 
hyperelliptic fibrations and gave the upper bound 
of the order of a finite subgroup of $\Rm{Aut} (f)$ in terms of $g$, $g(B)$ (the genus of $B$) and $K_f^2$.
\Red{
In [1], we considered the case that $h=0$, and
gave an upper bound on the order of $\mathrm{Aut}(f)$, generalizing 
results due to Arakawa \cite{Ara} and Chen \cite{Chen} for
hyperelliptic fibrations.}
\Red{We} constructed an example that shows \Red{our} bound is almost optimal in some cases.

\Red{
In this paper, we study the case that $h=1$, that is, cyclic covering 
fibrations of elliptic surfaces.
In order to state our results, we need some further notation.
Keeping the notation in Definition 0.1 with $h=1$, let $\varphi:W\to B$ 
be the relatively minimal model of
$\Ti{\varphi}: \Ti{W}\to B$, $R$ the image of $\Ti{R}$ by the 
natural birational map $\Ti{W}\to W$,
and $\Gamma_p$ the fiber of $\varphi$ over $p\in B$. If $\Gamma_p$ is a 
smooth
fiber, fixing a point $O_p\in \Gamma_p$, we can regard $\Gamma_p$ as an abelian group with the unit element $O_p$.
Put
\[
  \delta:=\min \{\sharp \mathrm{Aut}(\Gamma_p, O_p)\mid p\in B \text{ 
and }\Gamma_p \text{ smooth }\},
\]
where $\mathrm{Aut}(\Gamma_p, O_p):=\{\kappa \in 
\mathrm{Aut}(\Gamma_p)\mid \kappa(O_p)=O_p\}$.
}
\Blue{
\begin{thm}[Theorem~\ref{main1}]
Let $f:S \to B$ be a non-locally trivial primitive cyclic covering fibration of type $(g,1,n)$ 
with $g \geq \Rm{max}\{\frac{30n^2 -47n +25}{n+1}, \frac{7}{2}n(n-1)+1\}$.
Put 
\begin{align*}
\mu_{n}:=&\frac{12 n^2 \delta}{n^2 -1}.
\end{align*}
Assume furthermore that when $g(B)=0$, $f$ has at least $3$ singular fibers.
Let $G$ be a finite subgroup of $\Rm{Aut} (f)$.
Then it holds
\[
  \sharp G \leq \begin{cases}
    6(2g(B)-1)\mu_{n}K_{f}^2 & (g(B) \geq 1 ),\\
     5\mu_{n}K_{f}^2 & (g(B)=0).
  \end{cases}
\]
\end{thm}
\begin{thm}[Theorem~\ref{main2}]
Let $f:S \to B$ be a non-locally trivial primitive cyclic covering fibration of type $(g,1,n)$ 
with $g \geq \Rm{max}\{\frac{30n^2 -47n +25}{n+1},  \frac{7}{2}n(n-1)+1\}$.
Put 
\begin{align*}
\mu_{n}':=\frac{6 n^{2} \delta }{(n-1)(5n-4)}.
\end{align*}
Assume furthermore that 
the branch locus $R$ has singular points on at least three $($resp. one$)$ fibers 
when $g(B)=0$ $($resp. $g(B)\geq 1$$)$.
Let $G$ be a finite subgroup of $\Rm{Aut} (f)$.
Then it holds
\[
  \sharp G \leq \begin{cases}
    6(2g(B)-1)\mu_{n}'K_{f}^2 & (g(B) \geq 1 ),\\
     5\mu_{n}' K_{f}^2 & (g(B)=0).
  \end{cases}
\]
\end{thm}
\begin{thm}[Corollary~\ref{biell}]
Let $f:S \to B$ be a non-locally trivial bielliptic fibration 
with $g \geq 17$.
Assume furthermore that the branch locus $R$ has singular points on at least three $($resp. one$)$ fibers 
when $g(B)=0$ $($resp. $g(B)\geq 1$$)$.
Let $G$ be a finite subgroup of $\Rm{Aut} (f)$.
Then it holds
\[
  \sharp G \leq \begin{cases}
    24\delta(2g(B)-1)K_{f}^2 & (g(B) \geq 1 ),\\
     20 \delta  K_{f}^2 & (g(B)=0).
  \end{cases}
\]
\end{thm}
\begin{thm}[Corollary~\ref{main3}]
Let $f:S \to B$ be a non-locally trivial primitive cyclic covering fibration of type $(g,1,n)$ 
with $g \geq \Rm{max}\{\frac{30n^2 -47n +25}{n+1},  \frac{7}{2}n(n-1)+1\}$.
Put $\Rm{Aut} (S/B):=\{(\kappa_{S},\Rm{id}_{B})\in \Rm{Aut}(f)\}$.
Assume that the branch locus $R$ has a singular point.
Then it holds
\[
  \sharp \Rm{Aut} (S/B) \leq \frac{6 n^{2} \delta }{(n-1)(5n-4)}K_{f}^2. \]
\end{thm}
}
\Red{We state the outline of the proof}.
\Red{Let $G$ be a finite subgroup of $\Rm{Aut}(f)$}.
Then $G$ can be expressed as the extension of its horizontal part $H$ by its vertical part $K$, that is, 
$1\to K\to G\to H \to 1$ (exact). 
We can study $H$ by using known results for the automorphism groups of curves. 
Hence it suffices to \Blue{estimate} $K$.
Then $K$ induces a subgroup $\Ti{K}$ of the automorphism group of the elliptic surface \Red{$\varphi:W \to B$} preserving \Red{the branch locus $R$} and our task is reduced to estimating the order of $\Ti{K}$.
For this purpose, we use the localization of the invariant $K_f^2$.
It is known that $K_f^2$ can be localized to fibers of $f$ (\cite{Eno2}).
We refine the localized $K_f^2$ further by using the quantity defined at a point on a fiber, 
and obtain the new expression of it in Proposition~\ref{localization}.
Then, we estimate the order of $\Ti{K}$ from above with the localized $K_f^2$ multiplied by an explicit function in $g$ and $n$ (Propositions~\ref{estimateK2}, \ref{estimateK1} and \ref{estimateK1.5}).

In \Red{Section~1}, we recall \Red{the basic properties} of primitive cyclic covering fibrations mainly due to \cite{Eno}.
In Section~2, we recast Section~5 of \cite{Eno} and consider the refined localization of $K_f^2$.
In Section~3, we \Red{summarize} \Red{the basic properties} of elliptic surfaces as group manifolds.
In Section~4, \Red{we descend the automorphism group to the elliptic surface $W$ and analyze its action.}
In Section~5, we estimate the order of $\Ti{K}$.
In Section~6, we show our main result, Theorem~\ref{main1}, \ref{main2}, Corollary~\ref{biell} and \ref{main3}.
Section~7 is devoted to \Red{constructing} an example of bielliptic fibrations
with a large automorphism group.

\textbf{Acknowledgememt}

The author is deeply grateful to Professor Kazuhiro Konno for his valuable advices and supports.
The author is particularly grateful to Professor Osamu Fujino and Professor Kento Fujita for their continuous supports.  
The author also thanks to Doctor Makoto Enokizono for his precious advices.
The research is supported by JSPS KAKENHI No. 20J20055.

\tableofcontents

\section{Primitive cyclic covering fibrations}

We recall \Red{the basic} properties of primitive cyclic covering fibrations, most of which can be found in \cite{Eno}.

\begin{define}
Let $f:S \to B$ be a relatively minimal fibration of genus $g\geq2$.
We call it a {\it primitive cyclic covering fibration} of type $(g,1,n)$, when 
there are  a $($not necessarily relatively minimal $)$ fibration $\tilde{\varphi}:\Ti{W} \to B$ of 
genus $1$ ({\it i.e.} elliptic surface) and a classical $n$-cyclic covering
$$
\tilde{\theta}:\Ti{S}=
\mathrm{Spec}_{\Ti{W}}\left(\bigoplus_{j=0}^{n-1} \mathcal{O}_{\Ti{W}}(-j\Ti{\Fr{d}})\right)\to\Ti{W}
$$
branched over a smooth (not \Blue{necessarily} irreducible) curve $\Ti{R} \in |n\Ti{\Fr{d}}|$ 
and $\Ti{\Fr{d}} \in \Rm{Pic}(\Ti{W})$ such that $f$ is the relatively minimal model of 
$\tilde{f}=\tilde{\varphi}\circ\tilde{\theta}$.
\end{define}

In addition, we employ the following notations. 
We denote by $\Sigma=\langle \tilde{\sigma} \rangle$ the covering transformation group of $\tilde{\theta}$ and by 
$\varphi: W\to B$ \Blue{the relatively} minimal model of $\tilde{\varphi}:\Ti{W} \to B$ with the natural contraction map 
$\tilde{\psi}: \Ti{W} \to W$. 
Furthermore, 
$\widetilde{F}$, $F$, $\widetilde{\Gamma}$ and $\Gamma$ will denote general fibers of 
$\tilde{f}$, $f$,  $\tilde{\varphi}$ and $\varphi$, respectively.
\Blue{We write fibers of $\tilde{f}$, $f$,  $\tilde{\varphi}$ and $\varphi$ over $p$ as $\widetilde{F}_{p}$, $F_{p}$, $\widetilde{\Gamma}_{p}$ and $\Gamma_{p}$, respectively}. 

\begin{rmk}
\label{pcc}
We note important properties of primitive cyclic covering fibrations in the following, which can be found in \cite{Eno}. 
\begin{itemize}
\item There exists an automorphism $\sigma \in \Rm{Aut} (S)$ which satisfies the following:
\begin{itemize}
\item The natural morphism $\Ti{S}\to S$ is a minimal succession of blowing-ups 
that resolves all isolated fixed points of $\sigma$.
\item The automorphism $\Ti{\sigma}$ \Blue{is induced} by $\sigma$. 
\end{itemize}
\item Let $\Ti{R}_{v}$ be \Blue{the sum of $\Ti{\varphi}$-vertical} components of $\Ti{R}$.
The self-intersection number of each irreducible \Blue{component} of $\Ti{R}_{v}$ is \Blue{equals
to $-an$} for some positive integer $a$.  
\item \Blue{Any $\Ti{\varphi}$-vertical $(-1)$ curve intersects $\Ti{R}$}.
\end{itemize}
\end{rmk}

From now on, we let $f:S \to B$ be a  primitive cyclic covering fibration of type $(g,1,n)$ and freely use the above notations and conventions.

Since the restriction map 
$\tilde{\theta}|_{\widetilde{F}}:\widetilde{F} \to \widetilde{\Gamma}$ is a classical 
$n$-cyclic covering branched over $\Ti{R} \cap\widetilde{\Gamma}$, 
the Hurwitz formula for $\tilde{\theta}|_{\widetilde{F}}$ gives us
\begin{equation}
r:=\Ti{R} \widetilde{\Gamma}=\frac{2\bigl(g-1\bigr)}{n-1}.
\end{equation}
From $\Ti{R} \in |n\Ti{\Fr{d}}|$, it follows that $r$ is a multiple of $n$.

Since $\tilde{\psi}: \Ti{W} \to W$ is a composite of blowing-ups, we can write 
$\tilde{\psi}=\psi_{1}\circ\cdots\psi_{N}$, 
where $\psi_{i}:W_{i}\to W_{i-1}$ denotes the blowing-up 
at $z_{i} \in W_{i-1}\;(i=1,\cdots, N)$, $W_{0}=W$ and $W_{N}=\Ti{W}$.
We define \Blue{reduced curves} $R_{i}$ inductively as $R_{i-1}=(\psi_{i})_{\ast}R_{i}$ 
starting from $R_{N}=\Ti{R}$ down to $R_{0}=R$.
We call $R_{i}$ {\it \Blue{the branch locus} on $W_{i}$}.
We also put $E_{i}=\psi_{i}^{-1}(z_{i})$ and $m_{i}=\mathrm{mult}_{z_{i}}R_{i-1}\;(i=1,\cdots, N)$.

\begin{lem}[\Blue{\cite{Eno}, Lemma~1.5}]
In the above situation, the following hold for any $i=1,\cdots, N$.
\label{lem1.2}
\begin{itemize}
\item[$(1)$] Either $m_{i} \in n\Z_{\geq1}$ or $n\Z_{\geq1}+1$, where $\Z_{\geq 1}$ is the set of positive integers. 
Furthermore, $m_{i}\in n\Z_{\geq1}$
if and only if $E_{i}$ is not contained in $R_{i}$.

\item[$(2)$] $R_{i}={\psi}_{i}^{\ast}R_{i-1}-n[\frac{m_{i}}{n}]E_{i}$, where $[t]$ denotes the 
greatest integer not exceeding $t$.

\item[$(3)$] There exists a $\Fr{d}_{i} \in \Rm{Pic}(P_{i})$ such that $\Fr{d}_{i}=\psi_{i}^{\ast}\Fr{d}_{i-1}-[\frac{m_{i}}{n}]E_{i}$
 and $R_{i}\sim n\Fr{d}_{i}$, $\Fr{d}_{N}=\Ti{\Fr{d}}$. 
\end{itemize}
\end{lem}

We say that a singular point of \Blue{$R_i$} is {\it of type $n\mathbb{Z}$} (resp. $n\mathbb{Z}+1$) if its multiplicity 
is in $n\mathbb{Z}_{\geq 1}$ (resp. $n\mathbb{Z}_{\geq 1}+1$).

\section{Localization of $K_{f}^2$}

We \Blue{recall} the argument in  Section~5 of \cite{Eno} and give localizations of $K_f^2$.
Let $f:S \to B$ be a primitive cyclic covering fibration of type $(g,1,n)$.

First of all, let us recall \Blue{the singularity indices introduced in} \cite{Eno}.
For any fixed $p\in B$, we consider all singular points 
of \Blue{the branch loci} $R=R_{0}, \cdots, R_{N-1}$ over $\Gamma_{p}$. 
For any positive integer $k$, we let $\alpha_{k}(\Gamma_{p})$ be the number of singular points of multiplicity either $kn$ or $kn+1$ among them.
We put $\alpha_{k}:=\sum_{p\in B}\alpha_{k}(\Gamma_{p})$ and call it the {\it $k$-th singularity index} of the fibration.
For an effective vertical divisor $T$ and $p \in B$, we denote \Blue{by $T(p)$ the biggest subdivisor of $T$ whose support is in the fiber over $p$}.
Then $T = \sum_{p \in B}T(p)$.
We sometimes write $\sharp T$ \Blue{as} the number of irreducible components of $T$.

Let $j_{b,a}(\Gamma_{p})$ be the number of irreducible curves with genus $b$ and self-intersection number $-an$ contained in $\Ti{R}_{v}(p)$.
 
 We introduce the following indices:
 \[
 j_{\bullet,a}(\Gamma_{p}):= \sum_{b \geq 0} j_{b,a}(\Gamma_{p}), \quad 
 j_{b,\bullet}(\Gamma_{p}):= \sum_{a \geq 0} j_{b,a}(\Gamma_{p}), \quad
j_{0,1} :=\sum_{p \in B} j_{0,1}(\Gamma_{p}).
 \]
Let $A$ be the sum of all $\Ti{\varphi}$-vertical $(-n)$-curves contained in $\Ti{R}$ and put $\Ti{R}_0 := \Ti{R} - A$. 
Then the $0$-th singularity index is defined by $\alpha_{0}:=(K_{\tilde{\varphi}}+\Ti{R}_0)\Ti{R}_0$.
Put
\begin{align*}
\alpha_0^{+} (\Gamma_{p})&:=    \Ti{R}_h \Ti{\Gamma}_{p}-\sharp(\Rm{Supp} (\Ti{R}_{h})\cap \Rm{Supp} (\Ti{\Gamma}_{p})),\\ 
 \alpha_{0}(\Gamma_{p})&:=\alpha_{0}^{+}(\Gamma_{p}) -2 \sum_{a \geq 2} j_{0,a}(\Gamma_{p}),
 \end{align*}
where $\Ti{R}_h$ is \Blue{the $\Ti{\varphi}$-horizontal} part of $\Ti{R}$.
We note that $\alpha_{0}=\sum_{p\in B}\alpha_{0}(\Gamma_{p})$.
We introduce some indices depending only \Blue{on} $\Gamma_{p}$.
\begin{itemize}
\item[] $\chi_{\varphi}(\Gamma_{p}):=e(\Gamma_{p})/12$ where
$e(\Gamma_{p})$ is the topological Euler number of $\Gamma_{p}$
\item[] $\nu(\Gamma_{p}):=1-1/l(\Gamma_{p})$ where $l(\Gamma_{p})$ is a multiplicity of $\Gamma_{p}$ 
\item[] $\nu := \sum_{p \in B} \nu (\Gamma_{p})$
\end{itemize}
Then it holds that $e_{\varphi}=\sum_{p\in B}e(\Gamma_{p})$ and $\chi_{\varphi}=\sum_{p \in B}\chi_{\varphi}(\Gamma_{p})$.
 
\begin{lem}[\Blue{\cite{Eno2}, Lemma~2.1}]
\label{Enolem}
Let $f:S\to B$ be a primitive cyclic covering fibration of type $(g,1,n)$.
Then it holds 
\begin{align*}
 K_{f}^2 &=\sum_{k\geq 1}\left((n^{2}-1)k-n \right)\alpha_{k} 
 +\frac{(n-1)^{2}}{n}(\alpha_{0}-2j_{0,1})
+\frac{n^{2}-1}{n}r(\chi_{\varphi}+\nu) + j_{0,1}.
\end{align*}
\end{lem}
Hence if we put 
\begin{align*}
 K_{f}^2 (\Gamma_{p}) = & \sum_{k\geq 1}\left((n^{2}-1)k-n \right)\alpha_{k} (\Gamma_{p}) 
 +\frac{(n-1)^{2}}{n}(\alpha_{0}(\Gamma_{p}) -2j_{0,1}(\Gamma_{p}) ) \\
 &+\frac{n^{2}-1}{n}r(\chi_{\varphi}(\Gamma_{p}) +\nu(\Gamma_{p}) ) + j_{0,1}(\Gamma_{p}).
 \end{align*}
 for $p \in B$, then we get $K_{f}^{2}=\sum_{p \in B} K_{f}^2 (\Gamma_{p})$.
 
 Let $\Ha{\varphi}:\Ha{W}\to B$ be any intermediate elliptic surface between $\Ti{W}$ and $W$, and regard $\Ti{\psi}:\Ti{W}\to W$ as the composite of 
the natural birational morphisms $\Ha{\psi}: \Ti{W}\to \Ha{W}$ and $\Ch{\psi}:\Ha{W}\to W$.
\[
  \xymatrix{
    \Ti{W} \ar[r]^{\Ha{\psi}} \ar[dr]_{\Ti{\varphi}} & \Ha{W} \ar[d]^{\Ha{\varphi}} \ar[r]^{\Ch{\psi}}& W \ar[ld]^{\varphi}\\
     & B &
}
\]
We put $\Ha{R}=\Ha{\psi}_*\Ti{R}$.
The fiber of $\Ha{\varphi}$ over $p \in B$ will be denoted by $\Ha{\Gamma}_{p}$.
 
Let $\Ha{\alpha}_{k}(\Gamma_{p})$ and $\Ch{\alpha}_{k}(\Gamma_{p})$ 
be the number of the singular points 
contributing to $\alpha_{k}(\Gamma_{p})$ appearing in $\Ha{\psi}$ and $\Ch{\psi}$, respectively. 
We note that the number of singular points of $\Ha{R}$ on $\Ha{\Gamma}_{p}$ is counted by $\Ha{\alpha}_{k}(\Gamma_{p})$.
Then we have $\alpha_k(\Gamma_p)=\Ha{\alpha}_{k}(\Gamma_{p})+\Ch{\alpha}_{k}(\Gamma_{p})$. 
  
\Blue{We note that an arbitrary irreducible component of $\Ti{R}_v(p)$ is a proper transform of either an exceptional curve appearing from $\psi: \Ti{W} \to W$ or an irreducible component of $\Gamma_{p}$.
Let $\Ha{j}_{0,a}(\Gamma_{p})$ ({\it resp.} $\Ch{j}_{0,a}(\Gamma_{p})$) be the number of irreducible components of  $\Ti{R}_v(p)$ contributing to $j_{0,a}(\Gamma_{p})$
which are proper transforms of exceptional curves appearing from $\Ha{\psi}$ ({\it resp.} $\Ch{\psi}$).
Let $j_{0,a}'(\Gamma_{p})$ be the number of irreducible components contributing 
to $j_{0,a}(\Gamma_{p})$ which are proper transforms of irreducible components of $\Gamma_{p}$.}
Then we have $j_{0,a}(\Gamma_p)=\Ha{j}_{0,a}(\Gamma_{p})+\Ch{j}_{0,a}(\Gamma_{p})+j_{0,a}'(\Gamma_{p})$. 
Put $\Ha{j}_{0,\bullet}(\Gamma_{p}):=\sum_{a\geq1}\Ha{j}_{0,a}(\Gamma_{p})$, 
$\Ch{j}_{0,\bullet}(\Gamma_{p}):=\sum_{a \geq 1}\Ch{j}_{0,a}(\Gamma_{p})$
and $j_{0,\bullet}'(\Gamma_{p}):=\sum_{a \geq 1}j_{0,a}'(\Gamma_{p})$.

Choose and fix $p \in B$ and $\Ha{z} \in \Ha{\Gamma}_p$.  
We consider the vertical part $\Ti{R}_v=\sum_{p\in B}\Ti{R}_v(p)$ of $\Ti{R}$ with respect to $\Ti{\varphi}:\Ti{W}\to B$.
We let $\Ti{R}_v(p)_{\Ha{z}}$ be the biggest subdivisor of $\Ti{R}_v(p)$ contracted to $\Ha{z}$ by $\Ha{\psi}$.
Note that we have $\Ti{R}_v(p)_{\Ha{z}}\neq 0$ only when there exists a singular point (of \Red{the branch loci}) \Blue{of $n\Z+1$ type}
and which is infinitely near to $\Ha{z}$ 
(including $\Ha{z}$ itself).
Note also that $\Ti{R}_v(p)_{\Ha{z}}$ is a disjoint union of non-singular rational curves each of which is a $(-an)$-curve for some positive integer $a$.

To decompose $\Ti{R}_v(p)_{\Ha{z}}$, we define a family $\{L_{i}\}_i$ consisting of vertical irreducible curves in $\Ti{R}_v(p)_{\Ha{z}}$ as follows.
\begin{itemize}
\item[(i)] Choose and fix a $(-1)$-curve $E_1$ over $\Ha{z}$ on \Blue{an elliptic} surface between $\Ti{W}$ and $\Ha{W}$, and 
let $L_1$ be the proper transform of $E_1$ on $\Ti{W}$.

\item[(ii)] For $i \geq 2$, $L_{i}$ is the proper transform of an exceptional $(-1)$-curve $E_i$ that is contracted to a point $x_i$ in $E_k$ or its proper transform for some $k<i$.
\end{itemize}
Put $\Fr{L}:=\{\{L_{i}\}_{i}\mid$ $\{L_{i}\}_{i}$ satisfies (i) and (ii)$\}$. 
If $\{L_{i}\}_{i} $ and $\{L'_{j}\}_{j}$ ($\{L_{i}\}_{i} $,$\{L'_{j}\}_{j} \in \Fr{L}$) have a common curve, one can show the union of them $\{L_{i},L'_{j}\}_{i,j}$ is in $\Fr{L}$.
We define a \Blue{partial order} $\{L_{i}\}_{i} \leq \{L'_{j}\}_{j}$ if $\{L_{i}\}_{i} \subset \{L'_{j}\}_{j}$. 
Then $(\Fr{L},\leq)$ is a partial order set.
Applying \Blue{Zorn's lemma}, one should check that every chain admits \Red{the upper bound}.
Hence there exist maximal elements \Blue{$\{L_{1,k}\}_{k=1}^{k_{1}}, \cdots ,\{L_{\eta_{\Ha{z}},k}\}_{k=1}^{k_{\eta_{\Ha{z}}}}$}, where $\eta_{\Ha{z}}$ is the number 
of maximal elements.
\Blue{Any two of $\{L_{1,k}\}_{k=1}^{k_{1}}, \cdots ,\{L_{\eta_{\Ha{z}},k}\}_{k=1}^{k_{\eta_{\Ha{z}}}}$ have no common curves.}
We put $D_{t} := \sum _{k = 1}^{k_t} L_{t,k}$ for $t=1, \cdots , \eta_{\Ha{z}}$.
%

We can describe $\Ti{R}_{v}(p)_{\Ha{z}}$ 
\Blue{by using (the above) $D_{t}$'s as follows}.
\Red{The divisor} $\Ti{R}_{v}(p)_{\Ha{z}}$ is \Blue{decomposed into} a disjoint sum \Blue{consisting} of such \Red{sums} uniquely.
 We denote as
\begin{align}
\label{6/16.1} 
\Blue{\Ti{R}_{v}(p)_{\Ha{z}}= D_{1}+ \cdots  +D_{\eta_{\Ha{z}}},  \quad D_{t} =\sum _{k = 1}^{k_t}L_{t,k}}.
\end{align}
Let $C_{t,k}$ be the exceptional $(-1)$-curve whose proper \Blue{transform on} $\Ti{W}$ is $L_{t,k}$.
We let $\iota^t (\Gamma_{p})_{\Ha{z}}$ and $\kappa^t (\Gamma_{p})_{\Ha{z}}$ denote the numbers of singular points of \Blue{the branch loci $R=R_{0}, \cdots, R_{N-1}$} over $\Ha{z}$ of types $n\Z$ and $n\Z +1$, respectively, 
at which the proper transforms of two \Blue{curves in} $\{C_{t,k}\}_{k=1}^{k_t}$ meet, 
and put $\iota(\Gamma_p)_{\Ha{z}}=\sum_{t=1}^{\eta_{\Ha{z}}}\iota^t(\Gamma_p)_{\Ha{z}}$ and 
$\kappa(\Gamma_p)_{\Ha{z}}=\sum_{t=1}^{\eta_{\Ha{z}}}\kappa^t(\Gamma_p)_{\Ha{z}}$.
By the definition \Red{of $\iota^t (\Gamma_{p})_{\Ha{z}}$ and $\kappa^t (\Gamma_{p})_{\Ha{z}}$}, we note that $\iota(\Gamma_p)_{\Ha{z}}=\kappa(\Gamma_p)_{\Ha{z}}=0$ if $\Ha{R}$ is smooth at $\Ha{z}$.

Furthermore, we localize some indices to \Blue{points} on $\Ha{\Gamma}_{p}$.
Let $\Ha{\alpha}_{k}(\Gamma_{p})_{\Ha{z}}$ be the number of singular points contributing to $\Ha{\alpha}_{k}(\Gamma_{p})$ which is infinitely near to $\Ha{z}$.
\Red{Let $\Ha{j}_{0,a}(\Gamma_{p})_{\Ha{z}}$ ({\it resp.} $\Ha{j}_{0,a}^{t}(\Gamma_{p})_{\Ha{z}}$ for $t=1, \cdots \eta_{\Ha{z}}$) be the number of irreducible components of  $\Ti{R}_{v}(p)_{\Ha{z}}$ ({\it resp.} $D_t$) contributing to $\Ha{j}_{0,a}$}.
Put $\Ha{j}_{0,\bullet}(\Gamma_{p})_{\Ha{z}} :=\sum_{a \geq 1} \Ha{j}_{0,a}(\Gamma_{p})_{\Ha{z}}$ and 
\Red{$\Ha{j}_{0,\bullet}^{t}(\Gamma_{p})_{\Ha{z}} :=\sum_{a \geq 1} \Ha{j}_{0,a}^{t}(\Gamma_{p})_{\Ha{z}}$
for each $t=1, \cdots \eta_{\Ha{z}}$}.
\Red{Note that $\Ha{j}_{0,\bullet}(\Gamma_{p})_{\Ha{z}} = \sum_{t} \Ha{j}_{0,a}^{t}(\Gamma_{p})_{\Ha{z}}$}.
We have 
$\Ha{\alpha}_{k}(\Gamma_{p})=\sum_{\Ha{z}\in \Gamma_{p}}\Ha{\alpha}_{k}(\Gamma_{p})_{\Ha{z}}$ and
 $\Ha{j}_{0,a}(\Gamma_{p})=\sum_{\Ha{z}\in \Gamma_{p}}\Ha{j}_{0,a}(\Gamma_{p})_{\Ha{z}}$.
  We put
\begin{equation}\label{6/10,1}
\begin{aligned}
\alpha_0^{+} (\Gamma_{p})_{\Ha{z}} &:=
  \begin{cases}
    (\Ti{R}_{h}, \Ti{\Gamma}_{p,\Ha{z}})-\sharp \left(\Rm{Supp} (\Ti{R}_{h})\cap \Rm{Supp} (\Ti{\Gamma}_{p,\Ha{z}})\right) &({\rm if\;}\Ha{R} {\rm\;is\;singular\;at\;} \Ha{z}\;),\\
    \\
    \left({\rm The\;ramification\;index\;of\;}\Ha{\varphi}|_{\Ha{R}_{h}}:\Ha{R}_{h} \to B{\rm \;at\;}\Ha{z}\right)-1 &({\rm if\;}\Ha{R} {\rm\;is\;smooth\;at\;} \Ha{z}\;),\\
    \\
   0 & ({\rm if\;} \Ha{z} \not\in \Ha{R}),
  \end{cases}\\
\Ha{\alpha}_{0}(\Gamma_{p})_{\Ha{z}}&:= \alpha_{0}^{+}(\Gamma_{p})_{\Ha{z}}-2\sum_{a \geq 2} \Ha{j}_{0,a}(\Gamma_{p})_{\Ha{z}},
\end{aligned}
\end{equation}
where $\Ti{\Gamma}_{p,\Ha{z}}$ is the biggest subdivisor of $\Ti{\Gamma}_{p}$ which is contracted to $\Ha{z}$ by $\Ha{\psi}$.

We put 
 \begin{align*}
 K_{f}^2 (\Gamma_{p})_{\Ha{z}} := & \sum_{k\geq 1}\left((n^{2}-1)k-n \right)\Ha{\alpha}_{k} (\Gamma_{p}) _{\Ha{z}} 
 +\frac{(n-1)^{2}}{n}(\Ha{\alpha}_{0}(\Gamma_{p})_{\Ha{z}}  -2\Ha{j}_{0,1}(\Gamma_{p})_{\Ha{z}} ) + \Ha{j}_{0,1}(\Gamma_{p})_{\Ha{z}}.
\end{align*}

\smallskip

Then we get the following:
 \begin{prop}
 \label{localization}
Let $f:S\to B$ be a primitive cyclic covering fibration of type $(g,1,n)$.
Let $\Ti{\psi}=\Ch{\psi}\circ\Ha{\psi} $ be an arbitrary decomposition of $\Ti{\psi}:\Ti{W} \to W$ \Blue{so that} the commutative diagram
\[
  \xymatrix{
    \Ti{W} \ar[r]^{\Ha{\psi}} \ar[dr]_{\Ti{\varphi}} & \Ha{W} \ar[d]^{\Ha{\varphi}} \ar[r]^{\Ch{\psi}}& W\ar[ld]^{\varphi}\\
     & B. &
}
\]
Denote a fiber of $\Ha{\varphi}$ over $p \in B$ by $\Ha{\Gamma}_{p}$.
Then it holds that 
\begin{align*}
 K_{f}^2 (\Gamma_{p}) = &\sum_{\Ha{z}\in \Ha{\Gamma}_{p}} K_{f}^2 (\Gamma_{p})_{\Ha{z}} 
 + \sum_{k\geq 1}\left((n^{2}-1)k-n \right)\Ch{\alpha}_{k} (\Gamma_{p}) 
 -2\frac{(n-1)^{2}}{n}\left(\Ch{j}_{0,\bullet}(\Gamma_{p}) +j_{0,\bullet}'(\Gamma_{p})\right)\\
 &+\frac{n^{2}-1}{n}r(\chi_{\varphi}(\Gamma_{p}) +\nu(\Gamma_{p}) ) + \Ch{j}_{0,1}(\Gamma_{p})+j_{0,1}'(\Gamma_{p}).
 \end{align*}
 \end{prop}
\begin{proof}
 \Blue{By definition}, we have 
 \begin{align}
 \label{localalphaj}
 \alpha_{k}(\Gamma_{p})=\sum_{\Ha{z}\in \Ha{\Gamma}_{p}} \Ha{\alpha}_{k}(\Gamma_{p})_{\Ha{z}}
 +\Ch{\alpha}_{k}(\Gamma_{p}),\quad 
 j_{0,a}(\Gamma_{p})=\sum_{\Ha{z}\in \Ha{\Gamma}_{p}} \Ha{j}_{0,a}(\Gamma_{p})_{\Ha{z}}
 +\Ch{j}_{0,a}(\Gamma_{p})
 +j_{0,a}'(\Gamma_{p}).
 \end{align}
 By (\ref{localalphaj}), we have 
 \begin{align*}
  K_{f}^2 (\Gamma_{p}) = & \sum_{k\geq 1}\left((n^{2}-1)k-n \right)\alpha_{k} (\Gamma_{p}) 
 +\frac{(n-1)^{2}}{n}(\alpha_{0}(\Gamma_{p}) -2j_{0,1}(\Gamma_{p}) ) \\
 &+\frac{n^{2}-1}{n}r(\chi_{\varphi}(\Gamma_{p}) +\nu(\Gamma_{p}) ) + j_{0,1}(\Gamma_{p})  \\
 =& \sum_{k\geq 1}\left((n^{2}-1)k-n \right)\left(\sum_{\Ha{z}\in \Ha{\Gamma}_{p}} \Ha{\alpha}_{k}(\Gamma_{p})_{\Ha{z}}+\Ch{\alpha}_{k}(\Gamma_{p})\right)\\
 &+\frac{(n-1)^{2}}{n}\biggl(
 \sum_{\Ha{z}\in \Ha{\Gamma}_{p}} \Ha{\alpha}_{0}(\Gamma_{p})_{\Ha{z}}
 -2\sum_{a\geq 2}\left(\Ch{j}_{0,a}(\Gamma_{0})+j_{0,a}'(\Gamma_{p})\right)\biggr.\\
 &\biggl.-2\sum_{\Ha{z}\in \Ha{\Gamma}_{p}} \Ha{j}_{0,1}(\Gamma_{p})_{\Ha{z}}
 -2\left(\Ch{j}_{0,1}(\Gamma_{p})+j_{0,1}'(\Gamma_{p})\right)
 \biggr) \\
 &+\sum_{\Ha{z}\in \Ha{\Gamma}_{p}} \Ha{j}_{0,1}(\Gamma_{p})_{\Ha{z}}
 + \Ch{j}_{0,1}(\Gamma_{p}) + j'_{0,1}(\Gamma_{p})
+\frac{n^{2}-1}{n}r(\chi_{\varphi}(\Gamma_{p}) +\nu(\Gamma_{p}) ) \\
=&\sum_{\Ha{z}\in \Ha{\Gamma}_{p}} K_{f}^2 (\Gamma_{p})_{\Ha{z}} 
 + \sum_{k\geq 1}\left((n^{2}-1)k-n \right)\Ch{\alpha}_{k} (\Gamma_{p}) \\
 &-2\frac{(n-1)^{2}}{n}\left(\Ch{j}_{0,\bullet}(\Gamma_{p}) +j_{0,\bullet}'(\Gamma_{p})\right)\\
 &+\frac{n^{2}-1}{n}r(\chi_{\varphi}(\Gamma_{p}) +\nu(\Gamma_{p}) ) + \Ch{j}_{0,1}(\Gamma_{p})+j_{0,1}'(\Gamma_{p})
\end{align*}
 \end{proof}
 
 We consider and recast Lemma~5.2 in \cite{Eno}.

\begin{lem}
\label{lem5.2}
The following hold:
\smallskip

$(1)$  $\iota (\Gamma_{p})_{\Ha{z}}=\Ha{j}_{0,\bullet}(\Gamma_{p})_{\Ha{z}}-\eta_{\Ha{z}}$.
\smallskip

$(2)$ $\alpha_{0}^{+}(\Gamma_{p})_{\Ha{z}} \geq (n-2)
\left( \Ha{j}_{0,\bullet}(\Gamma_{p})_{\Ha{z}}-\eta_{\Ha{z}}+2\kappa (\Gamma_{p})_{\Ha{z}} \right)$.
\smallskip

$(3)$ $\displaystyle{\sum_{k \geq 1}\Ha{\alpha}_{k}(\Gamma_{p})_{\Ha{z}}\geq \sum _{a\geq 1}(an-2)\Ha{j}_{0,a} (\Gamma_{p})_{\Ha{z}}+2\eta_{\Ha{z}}-\kappa (\Gamma_{p})_{\Ha{z}}}$.
\end{lem}

\begin{proof}
For each $t$, we consider the graph $\Bb{G}_t$ corresponding to $D_t$ as follows. 
The vertex set $V(\Bb{G}_t)$ and the edge set $E(\Bb{G}_t)$ are respectively the sets of symbols 
$\{ v_{t,k}\}_{k=1}^{\Ha{j}_{0,\bullet}^{t}(\Gamma_{p})_{\Ha{z}}}$ and $\{e_{x}\}_{x}$, where $x$ runs over all the singular points 
contributing to $\iota^t (\Gamma_{p})_{\Ha{z}}$. 
If $C_{t,k}$ or a proper transform of it meets a proper transform of $C_{t,k'}$ at the singular point $x$ of type $n\Z$, 
the edge $e_{x}$ connects $v_{t,k}$ and $v_{t,k'}$. 
By the definition of $D_t$, we see that $\Bb{G}_t$ is a connected tree.
Counting the numbers of vertices and edges, we have  $\iota^t (\Gamma_{p})_{\Ha{z}}=\Ha{j}_{0,\bullet}^t(\Gamma_{p})_{\Ha{z}}-1$.
Thus, we get (1) by summing it up for $t$.

When the component $L_{t,k}$ of $D_t$ is a $(-an)$-curve, it is obtained by blowing $C_{t,k}$ up 
 $an-1$ times.
\Blue{We recall that $\Ha{j}_{0,a}^{t}(\Gamma_{p})_{\Ha{z}}$ is the number of irreducible components 
of $D_t$ with self-intersection number $-an$.
By $D_{t} := \sum _{k = 1}^{k_t} L_{t,k}$, we need blowing-ups $\sum_{a \geq 1} (an-1)\Ha{j}_{0,a}^{t}(\Gamma_{p})_{\Ha{z}}$ times to get $D_t$, disregarding overlaps.}
Taking into the account the duplication and the first blowing-up creating the $(-1)$-curve for $L_{t,1}$, 
we see that the number of blowing-ups to obtain $D_t$ is not less than
\[
 \sum_{a \geq 1} (an-1)\Ha{j}_{0,a}^{t}(\Gamma_{p})_{\Ha{z}}-\iota^t (\Gamma_{p})_{\Ha{z}}-\kappa^t (\Gamma_{p})_{\Ha{z}}+1\\
 =\sum_{a \geq 2} (an-2)\Ha{j}_{0,a}^{t}(\Gamma_{p})_{\Ha{z}}+2-\kappa^t (\Gamma_{p})_{\Ha{z}}, 
\]
since $\iota^t (\Gamma_{p})_{\Ha{z}}=\Ha{j}_{0,\bullet}^t(\Gamma_{p})_{\Ha{z}}-1$.
\Blue{By $\Ti{R}_{v}(p)_{\Ha{z}} = \sum_{t=1}^{\eta_{\Ha{z}}} D_t $, 
we need blowing-ups at least
\begin{align*}
 \sum_{t=1}^{\eta_{\Ha{z}}}\left(\sum_{a \geq 2} (an-2)\Ha{j}_{a}^{t}(\Gamma_{p})_{\Ha{z}}+2-\kappa^t (\Gamma_{p})_{\Ha{z}}\right)
 =\sum_{a \geq 2} (an-2)\Ha{j}_{0,a}(\Gamma_{p})_{\Ha{z}}+2\eta_{\Ha{z}}-\kappa(\Gamma_{p})_{\Ha{z}}
\end{align*}
times to get $\Ti{R}_{v}(p)_{\Ha{z}}$.}
This gives $(3)$.

It remains to show (2).
We may assume $\Ha{R}$ is singular at $\Ha{z}$.
Let $\Ti{\Gamma}_{p,\Ha{z}}=\sum m_i G_i$ be the irreducible decomposition. Then it follows from (\ref{6/10,1}) that 
\begin{align*}
\alpha_0 ^{+} (\Gamma_{p})_{\Ha{z}} = \sum_{i} m_{i} \Ti{R}_h G_i - \sharp(\Rm{Supp} (\Ti{R}_{h})\cap \Rm{Supp} ( \cup_{i} G_{i})) 
\geq \sum_{i} (m_{i}-1) \Ti{R}_h G_i.
\end{align*}
We consider a directed graph $\Bb{F}$ whose vertex set $V(\Bb{F})$ is the set of symbols $\{v_x\}$, where $x$ runs over all the singular points which contribute to 
either $\iota(\Gamma_{p})_{\Ha{z}}$ or $\kappa(\Gamma_{p})_{\Ha{z}}$.
We define the directed edge from $v_x$ to $v_{x'}$ if $x'$ is a singular point infinitely near to $x$ and any singular point between $x$ and $x'$ 
contributes to neither $\iota(\Gamma_{p})_{\Ha{z}}$ nor $\kappa(\Gamma_{p})_{\Ha{z}}$.
Let $\Bb{T}_{1}, \cdots, \Bb{T}_{s}$ be connected components of the graph $\Bb{F}$.
We note that $\Bb{T}_j$ ($j=1, \cdots, s$) is a \Blue{directed and rooted} tree graph.
We denote the leaf set of $\Bb{T}_j$ by $L(\Bb{T}_j)$.
By the definition of $\Bb{F}$, any vertex in $L(\Bb{T}_j)$ corresponds to a singular point contributing to $\iota(\Gamma_{p})_{\Ha{z}}$.
Let $x_{\Rm{last}}$ be a singular point of \Blue{the branch locus} over $x$ such that
there exist no singular \Blue{points} over $x_{\Rm{last}}$
and let $\Ti{E}^x$ be \Blue{the exceptional curve of the blow-up at} $x_{\Rm{last}}$.
%
%
Since $\Ti{E}^x$ arises from a singular point of type $n\mathbb{Z}$, we see that $\Ti{R}\Ti{E}^x$ is a positive multiple of $n$, 
but it may be possible that $\Ti{E}^x$ meets two vertical components of $\Ti{R}$.
Thus we have $\Ti{R}_h  \Ti{E}^x \geq n-2$.
Letting $m^x$ be the multiplicity along $\Ti{E}^x$ of $\Ti{\Gamma}_p$, we have 
\begin{align*}
 \sum_{i} (m_{i}-1) \Ti{R}_h G_i \geq \sum_{j=1}^{s}\sum_{x \in L(\Bb{T}_j)}(m^x -1)\Ti{R}_h  \Ti{E}^x.
\end{align*}

We will show that 
\begin{align*}
\sum_{x \in L(\Bb{T}_j)}(m^x -1) \geq \iota(\Bb{T}_j)+ 2\kappa(\Bb{T}_j),
\end{align*}
where $ \iota(\Bb{T}_j)$ (resp. $ \kappa(\Bb{T}_j)$) denotes the number of singular points of type $n\Z$ (resp. $n\Z+1$) in 
$\Bb{T}_j$.
Let $P^x(\Bb{T}_j)$ be the set consisting of vertices appearing in 
the path connecting the root of $\Bb{T}_j$ and $x\in  L(\Bb{T}_j)$. 
We denote the number of singular points of type $n\Z$ (resp. $n\Z+1$) in 
$P^x(\Bb{T}_j)$ by $ \iota(P^x(\Bb{T}_j))$ (resp. $ \kappa(P^x(\Bb{T}_j))$).
We claim that $m^x \geq 2\iota(P^x(\Bb{T}_j))+\kappa(P^x(\Bb{T}_j))$.
This can be seen as follows.
Put $P^x(\Bb{T}_j)=\{v_1, \cdots, v_l = v_x \}$ and let $m_1, \cdots, m_l$ be the multiplicities of the fiber $\Ti{\Gamma}_p$ along the proper transforms 
of the exceptional curves arising from $v_1, \cdots, v_l$, respectively.
If $v_1$ is of type $n\Z$, then we have $m_2\geq m_1 +2$.
If $v_1$ is of type $n\Z+1$, then we have $m_2\geq m_1 +1$.
So we get $m^x=m_l \geq 2\iota(P^x(\Bb{T}_j))+\kappa(P^x(\Bb{T}_j))$ inductively.
Then, 
\begin{align*}
\sum_{x \in L(\Bb{T}_j)}(m^x -1) \geq &\sum_{x \in L(\Bb{T}_j)} \left(2\iota(P^x(\Bb{T}_j))+\kappa(P^x(\Bb{T}_j))\right) -\sharp L(\Bb{T}_j)\\
\geq & 2\iota(\Bb{T}_j)+ 2\kappa(\Bb{T}_j) -\sharp L(\Bb{T}_j)\\
\geq & \iota(\Bb{T}_j)+ 2\kappa(\Bb{T}_j)
\end{align*}
as wished.

Now, we get 
\begin{align*}
\sum_{x \in L(\Bb{T}_j)}(m^x -1) \Ti{R}_h G_i
\geq &\sum_{j=1}^{s}\sum_{x \in L(\Bb{T}_j)}(m^x -1)\Ti{R}_h  \Ti{E}^x  \\
\geq & (n-2) \sum_{j=1}^{s}\sum_{x \in L(\Bb{T}_j)}(m^x -1) \\
\geq & (n-2)( \iota(\Gamma_{p})_{\Ha{z}}+ 2\kappa(\Gamma_{p})_{\Ha{z}}).
\end{align*}
\Blue{Plugging the above inequalities} to (\ref{6/10,1}), 
we conclude that $\alpha_0 ^{+} (\Gamma_{p})_{\Ha{z}}\geq (n-2)( \iota(\Gamma_{p})_{\Ha{z}}+ 2 \kappa(\Gamma_{p})_{\Ha{z}})$.
Then, since $\iota(\Gamma_{p})_{\Ha{z}}=j(\Gamma_{p})_{\Ha{z}}-\eta_{\Ha{z}}$ by (1), we get (2).
\end{proof}

We can show the following inequality similarly to Lemma~4.7 of \cite{Eno2}
\begin{lem}
\label{Eno4.7}
If $n=2$, then it holds
\begin{align*}
\kappa(\Gamma_{p})_{\Ha{z}} \leq \frac{2}{3}\sum_{a\geq2} (a-1)\Ha{j}_{0,a}(\Gamma_{p})_{\Ha{z}}.
\end{align*}
\end{lem}

 \begin{lem}
  \label{alphabound}
 Let the notation and the assumption as above.
 Then it holds that 
 \begin{align*}
\Ha{\alpha}_{0}(\Gamma_{p})_{\Ha{z}}+B_{n}\sum_{k \geq 1}\Ha{\alpha}_{k}(\Gamma_{p})_{\Ha{z}}-
  \begin{cases}
  nB_{n}\Ha{j}_{0,1}(\Gamma_{p})_{\Ha{z}} &(n\geq3)\\
  \\
  0 &(n=2)\\
  \end{cases}
\end{align*}
\Blue{is non-negative},
where $B_{n}=2/n$ if $n \geq 4$, $B_{3}=1/2$ and $B_{2}=3/2$.
 \end{lem}
 \begin{proof}
 If $n\geq 3$, we can show the desired inequality similarly to Lemma~2.2 of \cite{Aka}.
 
\Blue{For $n=2$, we have}
 \begin{align*}
  &\Ha{\alpha}_{0}(\Gamma_{p})_{\Ha{z}}+B_{2}\sum_{k \geq 1}\Ha{\alpha}_{k}(\Gamma_{p})_{\Ha{z}}\\
 =&\alpha^{+}_{0}(\Gamma_{p})_{\Ha{z}}-2\sum_{a \geq 2}\Ha{j}_{0,a}(\Gamma_{p})_{\Ha{z}}
 +B_{2}\sum_{k \geq 1}\Ha{\alpha}_{k}(\Gamma_{p})_{\Ha{z}}\\
 \geq&\alpha^{+}_{0}(\Gamma_{p})_{\Ha{z}}+\sum_{a \geq 2}
 \left(B_{2}(2a-2)-2 \right)\Ha{j}_{0,a}(\Gamma_{p})_{\Ha{z}}+2B_{2}\eta_{\Ha{z}}-B_{2}\kappa(\Gamma_{p})_{\Ha{z}}\\
 \geq &\alpha^{+}_{0}(\Gamma_{p})_{\Ha{z}}+\sum_{a \geq 2}
 \left(\frac{4}{3}(a-1)B_{2}-2 \right)\Ha{j}_{0,a}(\Gamma_{p})_{\Ha{z}},
 \end{align*}
\Blue{where the first inequality} is by (3) of Lemma~\ref{lem5.2},
 the second inequality is by Lemma~\ref{Eno4.7}.
Since $B_{2}=3/2$, \Blue{the coefficients} of $\Ha{j}_{0,a}(\Gamma_{p})_{\Ha{z}}$ are non-negative for $a \geq 2$.
 \end{proof}
 
 \begin{lem}
 \label{lowerbdK1}
 The following holds for $p \in B$:
 \begin{align*}
 K_{f}^2(\Gamma_{p})_{\Ha{z}} \geq \sum_{k \geq 1} \left( (n^2-1)k-n-\frac{(n-1)^{2}}{n}B_{n}\right)
 \Ha{\alpha}_{k}(\Gamma_{p})_{\Ha{z}}.
 \end{align*}
 In particular, $K_{f}^2(\Gamma_{p})_{\Ha{z}}$ is non-negative.
 \end{lem}
 \begin{proof}
 If $n \geq 4$, it is clear from Lemma~\ref{alphabound}.
 
 If $n=3$, we have 
 \begin{align*}
 \Ha{\alpha}_{0}(\Gamma_{p})_{\Ha{z}}-2\Ha{j}_{0,1}(\Gamma_{p})_{\Ha{z}}
  \geq -B_{3}\sum_{k \geq 1}\Ha{\alpha}_{k}(\Gamma_{p})_{\Ha{z}}-\frac{1}{2}\Ha{j}_{0,1}(\Gamma_{p})_{\Ha{z}}
 \end{align*}
 from Lemma~\ref{alphabound}.
 Hence we get 
 \begin{align*}
 K_{f}^2(\Gamma_{p})_{\Ha{z}}& \geq \sum_{k \geq 1} \left( 8k-3\right)
 \Ha{\alpha}_{k}(\Gamma_{p})_{\Ha{z}} 
 -\frac{4}{3}B_{3}\sum_{k \geq 1} \Ha{\alpha}_{k}(\Gamma_{p})_{\Ha{z}}
- \frac{2}{3}\;\Ha{j}_{0,1}(\Gamma_{p})_{\Ha{z}}+\Ha{j}_{0,1}(\Gamma_{p})_{\Ha{z}}\\
&\geq  \left( 8k-3-\frac{4}{3}B_{3}\right)
 \Ha{\alpha}_{k}(\Gamma_{p})_{\Ha{z}}.
 \end{align*}
 If $n=2$, we have 
 \begin{align*}
 K_{f}^2(\Gamma_{p})_{\Ha{z}}& = \sum_{k \geq 1} \left( 3k-2\right)
 \Ha{\alpha}_{k}(\Gamma_{p})_{\Ha{z}}
 +\frac{1}{2}\Ha{\alpha}_{0}(\Gamma_{p})_{\Ha{z}}.
 \end{align*}
 Hence we get 
  \begin{align*}
 K_{f}^2(\Gamma_{p})_{\Ha{z}}& \geq \sum_{k \geq 1} \left( 3k-2-\frac{1}{2}B_{2}\right)
 \Ha{\alpha}_{k}(\Gamma_{p})_{\Ha{z}}
 \end{align*}
 from Lemma~\ref{alphabound}.
\end{proof}

\begin{lem}
\label{6/10,2}
 The following holds for $p \in B$.
 \begin{align*}
 K_{f}^2 (\Gamma_{p}) \geq &\sum_{\Ha{z}\in \Ha{\Gamma}_{p}} 
  \frac{(n-1)^{2}}{n}\alpha^{+}_{0}(\Gamma_{p})_{\Ha{z}}
 +\sum_{\Ha{z}\in \Ha{\Gamma}_{p}}
 \sum_{k\geq 1}\left((n^{2}-1)k-n-2\frac{(n-1)^2}{n} \right)\Ha{\alpha}_{k} (\Gamma_{p}) _{\Ha{z}}\\
  &+ \sum_{k\geq 1}\left((n^{2}-1)k-n \right)\Ch{\alpha}_{k} (\Gamma_{p}) 
 -2\frac{(n-1)^{2}}{n}\left(\Ch{j}_{0,\bullet}(\Gamma_{p})+j_{0,\bullet}'(\Gamma_{p})  \right)\\
 &+\frac{n^{2}-1}{n}r(\chi_{\varphi}(\Gamma_{p}) +\nu(\Gamma_{p}) ) + \Ch{j}_{0,1}(\Gamma_{p})+j_{0,1}'(\Gamma_{p}).
 \end{align*}
 \end{lem}
 
 \begin{proof}
 By \Red{the definition of $K_{f}^2 (\Gamma_{p})_{\Ha{z}}$}, we have 
 \begin{align*}
 K_{f}^2 (\Gamma_{p})_{\Ha{z}} = & \sum_{k\geq 1}\left((n^{2}-1)k-n \right)\Ha{\alpha}_{k} (\Gamma_{p}) _{\Ha{z}} 
 +\frac{(n-1)^{2}}{n}(\Ha{\alpha}_{0}(\Gamma_{p})_{\Ha{z}}  -2\Ha{j}_{0,1}(\Gamma_{p})_{\Ha{z}} ) + \Ha{j}_{0,1}(\Gamma_{p})_{\Ha{z}}
\end{align*}
for $\Ha{z}\in \Ha{\Gamma}_{p}$.
Since irreducible components of $\Ha{R}(p)_{\Ha{z}}$ arise from singular points
of type $n\Z+1$, we have 
$\Ha{j}_{0,\bullet}(\Gamma_{p})_{\Ha{z}} \leq \sum_{k \geq 1} \Ha{\alpha}_{k}(\Gamma_{p})_{\Ha{z}}$.
Hence we get
\begin{align*}
 K_{f}^2 (\Gamma_{p})_{\Ha{z}} \geq  \frac{(n-1)^{2}}{n}\alpha^{+}_{0}(\Gamma_{p})_{\Ha{z}}
 + \sum_{k\geq 1}\left((n^{2}-1)k-n-2\frac{(n-1)^2}{n} \right)\Ha{\alpha}_{k} (\Gamma_{p}) _{\Ha{z}}.
\end{align*}
\Blue{Combining this with} Proposition~\ref{localization}, we get the desired inequality.
 \end{proof}

\section{Automorphism groups of elliptic surfaces}

In this section, we summarize some properties of automorphism groups of elliptic surfaces.  
Let $\varphi: W \to B$ be a relatively minimal elliptic surface.
We fix a point $p \in B$ and 
denote by $\Delta \subset B$ an analytic open neighborhood of $p$.
We suppose that $\Gamma_{p}$ is not multiple.
Let $\varphi_{\Delta} : W_{\Delta} \to \Delta$ be a restriction of $\varphi:W \to B$ to $\Delta$:
\[\xymatrix{
 W_{\Delta} \ar[d]_{\varphi_{\Delta}} \ar@{^{(}->}[r] & W \ar[d]^{\varphi} \\
 \Delta \ar@{^{(}->}[r] & B. \\
}\]
Replacing $\Delta$ by a smaller neighborhood of $p$ if necessary,
we may assume that 
$\varphi_{\Delta}$ has only one singular fiber $\Gamma_{p}$.

Let $W_{\Delta,\Rm{sm}}:=W_{\Delta}\setminus \Rm{Sing}(\Gamma_{p})$ and 
$\Gamma_{p,\Rm{sm}}:=\Gamma_{p}\setminus \Rm{Sing}(\Gamma_{p})$,
where $\Rm{Sing}(\Gamma_{p})$ \Blue{is the set} of 
critical points of $\varphi_{\Delta}$ on $\Gamma_{p}$.
Then there exists the natural map $\varphi_{\Delta,\Rm{sm}}:W_{\Delta,\Rm{sm}} \to \Delta$.
\Red{The following theorem is due to Kodaira (\cite{Kod})}.
\begin{thm}[Kodaira]
\label{kodaira}
There exist three holomorphic maps
\begin{align*}
 O_{\Delta}:\Delta \to W_{\Delta,\Rm{sm}},\quad
 \mu_{\Delta}:W_{\Delta,\Rm{sm}}\times_{\Delta} W_{\Delta,\Rm{sm}} \to W_{\Delta,\Rm{sm}},\quad
 \iota_{\Delta}:W_{\Delta,\Rm{sm}} \to W_{\Delta,\Rm{sm}}
\end{align*}
over $\Delta$ which satisfy the following conditions.
\begin{itemize}
\item[(1)] The fiber germ $\varphi_{\Delta,\Rm{sm}}:W_{\Delta,\Rm{sm}} \to \Delta$ is
a group manifold over $\Delta$.
\item[(2)] Let $\Gamma_{q}$ $(q \in \Delta)$ be a fiber of $\varphi_{\Delta}$ and let $O_{q}:=\Gamma_{q}\cap O_{\Delta}(\Delta)$.
These three maps induce the commutative group law on $\Gamma_{q,\Rm{sm}}$
whose unit element is $O_{q}$.
The group $(\Gamma_{q,\Rm{sm}},O_{q})$ is one of the followings.
\[
(\Gamma_{q,\Rm{sm}},O_{q})\cong 
\begin{cases}
\text{The group law of elliptic curve} & (\Rm{I}_{0}\text{-type})\\
\Bb{G}_{m}\times G(\Gamma_{q}) & (\Rm{I}_{c}\text{-type}) \\
\Bb{G}_{a}\times G(\Gamma_{q}) & (\text{other})
\end{cases}
\]
where 
\begin{align*}
&G(\Rm{I}_{c})    \cong \Bb{Z}/c\Bb{Z}, \\
&G(\Rm{I}_{2c}^{\ast})      \cong    (\Bb{Z}/2\Bb{Z})^{2}, \\
&G(\Rm{I}_{2c+1}^{\ast}) \cong  \Bb{Z}/4\Bb{Z}, \\
&G(\Rm{II}) \cong G(\Rm{II}^{\ast}) \cong \{0 \},   \\
&G(\Rm{III}) \cong G(\Rm{III}^{\ast}) \cong \Bb{Z}/2\Bb{Z},\\
&G(\Rm{IV}) \cong G(\Rm{IV}^{\ast}) \cong \Bb{Z}/3\Bb{Z}.
\end{align*}
\end{itemize}
\end{thm}

\begin{proof}
Since $\Gamma_{p}$ is not multiple, there exists a projective elliptic surface over \Blue{a smooth} projective curve
$\varphi^{\flat} :W^{\flat} \to B^{\flat}$ which admits a global section and contains $\varphi_{\Delta} : W_{\Delta} \to \Delta$ as a fiber germ:
\[\xymatrix{
 W_{\Delta} \ar[d]_{\varphi_{\Delta}} \ar@{^{(}->}[r] & W^{\flat} \ar[d]^{\varphi^{\flat}} \\
 \Delta \ar@{^{(}->}[r] & B^{\flat} . \\
}\]
We denote by $O_{\flat}:B^{\flat} \to W^{\flat}$ a global section of $\varphi^{\flat}$. 
Let $W_{\Rm{sm}}^{\flat}$ be \Blue{an open} subset of $W^{\flat}$ consisting of all regular points of $\varphi^{\flat}$.
From Theorem~9.1 of \cite{Kod}, there \Red{exist} holomorphic maps
$\mu_{\flat}:W_{\Rm{sm}}^{\flat}\times_{B^{\flat}} W_{\Rm{sm}}^{\flat} \to W_{\Rm{sm}}^{\flat}$ and
$\iota_{\flat}:W_{\Rm{sm}}^{\flat}\to W_{\Rm{sm}}^{\flat}$ over $B^{\flat}$ which satisfy \Red{conditions} (1) and (2).
These three maps $O_{\flat}$, $\mu_{\flat}$ and $\iota_{\flat}$ induce $O_{\Delta}$, $\mu_{\Delta}$ and $\iota_{\Delta}$ by the universality of base change. 
\end{proof}
\begin{flushleft}
\begin{itemize}
\item We denote by $W_{\Rm{sm}}(\Delta)$ the set of sections of $\varphi_{\Delta,\Rm{sm}}$.
Since $\varphi_{\Delta,\Rm{sm}}:W_{\Delta,\Rm{sm}} \to \Delta$ is a group manifold over $\Delta$, 
$W_{\Rm{sm}}(\Delta)$ is \Blue{a commutative group} with \Red{the unit element} $O_{\Delta}$.
The symbol $\beta_{1}\oplus \beta_{2}$ denotes the sum of 
$\beta_{1},\: \beta_{2} \in W_{\Rm{sm}}(\Delta)$ and 
$\ominus \beta$ denotes the inverse element of $\beta \in W_{\Rm{sm}}(\Delta)$.
\item We have the following commutative diagram for a section $\beta \in W_{\Rm{sm}}(\Delta)$:
\[\xymatrix{
W_{\Delta,\Rm{sm}}   \ar[rr]_{\beta \times_{W_{\Delta,\Rm{sm}}}(\Rm{pr}_{1}) } \ar[d]_{\varphi_{\Delta}}  & & 
W_{\Delta,\Rm{sm}}\times_{\Delta} W_{\Delta,\Rm{sm}} \ar[rr]_{\Rm{pr}_{2}} \ar[d]^{\Rm{pr}_{1}} & &
W_{\Delta,\Rm{sm}} \ar[d]^{\varphi_{\Delta}}\\
\Delta \ar[rr]_{\beta}& & W_{\Delta,\Rm{sm}} \ar[rr]_{\varphi_{\Delta}}&  & \Delta \\
}\]
Hence we have an automorphism 
$\tau_{\beta}:=\mu_{\Delta} \circ \left(\beta \times_{W_{\Delta,\Rm{sm}}}(\Rm{pr}_{1})\right) \in \Rm{Aut} (W_{\Delta,\Rm{sm}}/\Delta)$
for each $\beta \in W_{\Rm{sm}}(\Delta)$.
Thus, there exists a natural injective homomorphism $W_{\Rm{sm}}(\Delta) \to \Rm{Aut} (W_{\Delta,\Rm{sm}}/\Delta);\beta \mapsto \tau_{\beta}$.
\item We define the following groups.
\begin{align}
\nonumber
 &\Rm{Aut} (W_{\Delta,\Rm{sm}}/\Delta):=
\{ \kappa \in \Rm{Aut} (W_{\Delta,\Rm{sm}}) \mid  \varphi_{\Delta,\Rm{sm}} \circ \kappa= \varphi_{\Delta,\Rm{sm}} \}\\
\nonumber
 &\Rm{Aut} (W_{\Delta,\Rm{sm}}/\Delta,O_{\Delta}):=
\{ \kappa \in \Rm{Aut} (W_{\Delta,\Rm{sm}}/\Delta) \mid \kappa \circ O_{\Delta} = O_{\Delta} \}\\
\label{Fixauto}
&\Red{\Rm{Aut} (\Gamma_{p,\Rm{sm}}, O_{p}):=
\{ \kappa \in \Rm{Aut} (\Gamma_{p,\Rm{sm}}) \mid   \kappa (O_{p})= O_{p} \}}
\end{align}
\end{itemize}
\end{flushleft}

\smallskip

\begin{prop}
There exists the exact sequence
\begin{align*}
1 \to W_{\Rm{sm}}(\Delta) \to \Rm{Aut} (W_{\Delta, \Rm{sm}}/\Delta) \to \Rm{Aut} (W_{\Delta,\Rm{sm}}/\Delta, O_{\Delta}) \to 1.
\end{align*}
In particular, it holds that 
\begin{align*}
\Rm{Aut} (W_{\Delta, \Rm{sm}}/\Delta) \cong W_{\Rm{sm}}(\Delta) \rtimes  \Rm{Aut} (W_{\Delta, \Rm{sm}}/\Delta, O_{\Delta}).
\end{align*}
\end{prop}

\begin{proof}
It is sufficient to show the following two statements.
\begin{itemize}
\item[(1)]There exist $\tau \in W_{\Rm{sm}}(\Delta)\subset \Rm{Aut} (W_{\Delta, \Rm{sm}}/\Delta)$ and
$\varepsilon \in \Rm{Aut} (W_{\Delta, \Rm{sm}}/\Delta, O_{\Delta})$ 
for an action $\kappa \in \Rm{Aut} (W_{\Delta, \Rm{sm}}/\Delta)$ such that $\kappa = \tau \circ \varepsilon$.

\item[(2)]$W_{\Rm{sm}}(\Delta) \lhd \Rm{Aut} (W_{\Delta, \Rm{sm}}/\Delta)$.
\end{itemize}

(1)~~An action  $\kappa \in \Rm{Aut} (W_{\Delta, \Rm{sm}}/\Delta)$ determines a section $\kappa\circ O_{\Delta} \in W_{\Rm{sm}}(\Delta)$.
We consider the action $\tau_{\ominus(\kappa \circ O_{\Delta})}$.
Then we have $\varepsilon:=\tau_{\ominus(\kappa \circ O_{\Delta})} \circ \kappa \in \Rm{Aut} (W_{\Delta, \Rm{sm}}/\Delta, O_{\Delta})$.
Hence we get $\tau_{\ominus(\kappa \circ O_{\Delta})}^{-1} \in W_{\Rm{sm}}(\Delta)$ and
$\varepsilon \in \Rm{Aut} (W_{\Delta, \Rm{sm}}/\Delta, O_{\Delta})$ 
such that $\kappa = \tau_{\ominus(\kappa \circ O_{\Delta})}^{-1} \circ \varepsilon $.

(2)~~We have to show $\kappa \circ \tau \circ \kappa^{-1} \in W_{\Rm{sm}}(\Delta)$
for any $\kappa \in \Rm{Aut} (W_{\Delta, \Rm{sm}}/\Delta)$ and $\tau \in W_{\Rm{sm}}(\Delta)$.
In order to show the above statement, it is sufficient to show 
$\varepsilon \circ \tau \circ \varepsilon^{-1} \in W_{\Rm{sm}}(\Delta) $
for any $\varepsilon \in \Rm{Aut} (W_{\Delta, \Rm{sm}}/\Delta,O_{\Delta})$ and $\tau \in W_{\Rm{sm}}(\Delta)$.
Suppose a section $\beta \in W_{\Rm{sm}}(\Delta)$ corresponds to the action $\tau \in \Rm{Aut} (W_{\Delta, \Rm{sm}}/\Delta)$ (namely $\tau = \tau_{\beta}$).
Let $\Gamma_{q}$ be an arbitrary smooth fiber of $\varphi_{\Delta}$.
Then it holds that $(\varepsilon\circ\tau_{\beta}\circ\varepsilon^{-1}) |_{\Gamma_{q}}= (\tau _{\varepsilon \circ \beta})|_{\Gamma_{q}}$ 
since $\varepsilon|_{\Gamma_{q}}$ is compatible with the group law of $\Gamma_{q}$.
Therefore we have $(\varepsilon\circ\tau_{\beta}\circ\varepsilon^{-1}) = \tau _{\varepsilon \circ \beta} \in  W_{\Rm{sm}}(\Delta)$.
\end{proof}

\section{Automorphism of fibered surface}

For a fibration $f:S \to B$, we define the automorphism group of $f$ as 
\[
\Rm{Aut}(f):=\{(\kappa_{S},\kappa_{B})\in \Rm{Aut} (S)\times \Rm{Aut}(B) \mid f \circ \kappa_{S} =  \kappa_{B} \circ  f \}.
\]

Let $f:S\to B$ be a primitive cyclic covering fibration of type $(g, 1, n)$.
Let $\Sigma$ be the cyclic group of order $n$ generated by $\sigma$,
where $\sigma$ is defined in Remark~\ref{pcc}.
Let $G$ be an arbitrary finite subgroup of $\Rm{Aut} (f)$.
Since $\Rm{Aut} (S/B):=\{(\kappa_{S}, \Rm{id})\in \Rm{Aut}(f)\}$ is a finite group, we may assume $\sigma \in G$ to estimate the order of it.
We have the exact sequence 
\begin{align*}
1\to K \to G \to H \to1, 
\end{align*}
where 
$K:=\{(\kappa_{S}, {\it id})\in G\}$ and  $H:=\{\kappa_{B} \in \Rm{Aut} (B)\mid  (\kappa_{S},\kappa_{B})\in G, \exists \kappa_{S}\in \Rm{Aut} (S)\}$.
\begin{lem}
Assume $r=\Ti{R}\Ti{\Gamma}\geq 4n$.
Take a smooth fiber $F$ of $f$ and a point $z \in F$.
Let $\kappa_{F}$ be an automorphism of $F$.
Then it holds that 
\begin{align*}
\kappa_{F}(\Sigma \cdot z)=\Sigma \cdot \kappa_{F}(z),
\end{align*}
where $\Sigma \cdot z$ denotes the $\Sigma$-orbits of $z$.
\end{lem}
\begin{proof}
The subgroup $\Sigma \subset \Blue{\Rm{Aut}(F)}$ induces the quotient map $\theta:F \to \Gamma := F/\Sigma$ of degree $n$. 
From the assumption $r \geq 4n$, there \Red{exists} \Blue{an isomorphism} $\kappa_{\Gamma}:\Gamma \to \Gamma$ such that the diagram
\[\xymatrix{
 F \ar[d]_{\theta} \ar[r]^{\kappa_{F}} & F \ar[d]^{\theta} \\
 \Gamma \ar[r]_{\kappa_{\Gamma}} & \Gamma \\
}\]
commutes for any $\kappa_{F}\in \Rm{Aut}(F)$ from Proposition~3.1 of 
\cite{Aka}.
Thus, we obtain $\kappa_{F}(\Sigma\cdot{z})=\Sigma\cdot{\kappa_{F}(z)}$.
\end{proof}
In what follows in section~4, we tacitly assume that $r \geq 4n$. 

We can show the following by a similar argument as in Lemma~3.3 of \cite{Aka}.
\begin{lem}
Let $F$ be a smooth fiber of $f:S \to B$.
Regard $\Sigma$ as a subgroup of $\Rm{Aut} (F)$.
Then $\Sigma$ is a normal subgroup of $\Rm{Aut} (F)$.
\end{lem}
Since we have $\Sigma \lhd K$,
the action of $K\subset \Rm{Aut} (S/B)$ on $S$ can be lifted to the one on $\Ti{S}$ and we can regard $K$ as a subgroup of $\Rm{Aut}(\Ti{S}/B)$.
If we put $\Ti{K}=K/\Sigma$, then we have the exact sequence  
\begin{align*}
1\to \Sigma \to K \to \Ti{K} \to 1.
\end{align*}
Note that $\Ti{K} \subset \Rm{Aut} (\Ti{W}/ B)$ and $\Ti{R}$ is $\Ti{K}$-stable (namely $\Ti{K}(\Ti{R})=\Ti{R}$).

\smallskip

\begin{lem}
The action of $\Ti{K}$ descends down faithfully on the relatively minimal model $\varphi:W \to B$.
Hence we can regard $\Ti{K}$ as a subgroup of $\Rm{Aut} (W/B)$.
\end{lem}
\begin{proof}
\Blue{Let $\Ti{\Gamma}_{p}$ be an arbitrary singular fiber of $\Ti{W} \to B$ and $E$ any $(-1)$-curve contained in $\Ti{\Gamma}_{p}$}.
Then its $\Ti{K}$-orbits $\Ti{K}\cdot E$ satisfies one of the following:
\begin{itemize}
\item $\Ti{K}\cdot E$ consists \Blue{of a disjoint} union of $(-1)$-curves in $\Ti{\Gamma}_{p}$.
\item We can find two curves in $\Ti{K}=\{E_{1},\cdots,E_{t}\}$ meeting a point.
\end{itemize}
In the former case, contracting $\Ti{K}\cdot E$ to points, the action of $\Ti{K}$ descends down to the action on the new fiber 
obtained by the contraction.

We will show the latter case does not occur.
Since the intersection form on $E_{1}\cup E_{2}$ is negative semi-definite, $E_{1}^2 = E_{2}^2 = -1$ and $E_{1}E_{2} > 0$,
the only possibility is:
$E_{1} E_{2}=1$ and $(E_{1}+E_{2})^2=0$.
So we get 
\begin{align*}
\Rm{Supp}\; \Gamma_{p}= E_{1}\cup E_{2}
\end{align*}
by Zariski's lemma. This is impossible by the classification of singular fibers of elliptic surfaces.
\end{proof}

\smallskip 

We fix a point $p \in B$ and denote by 
$\Delta \subset B$ an analytic open neighborhood of $p$.  
We suppose that $\Gamma_{p}$ is not multiple.
Let $\varphi_{\Delta}:W_{\Delta} \to \Delta$ be a restriction of $\varphi:W \to B$ to $\Delta$.
We note that there exists a natural inclusion $\Ti{K} \subset \Rm{Aut} (W_{\Delta,\Rm{sm}}/\Delta)$.
\begin{define}
We call $\kappa \in \Ti{K}$ \Blue{a {\it translation}} of $\Ti{K}$ if $\kappa$ satisfies
either $\kappa(C_{\Delta})\neq C_{\Delta}$ or $\kappa|_{C_{\Delta}}\neq \Rm{id}_{C_{\Delta}}$ for any $\varphi_{\Delta}$-horizontal local analytic branch $C_{\Delta}$ on $W_{\Delta,\Rm{sm}}$.
Let $T(\Ti{K})_{p}$ be \Blue{the set} consisting of \Blue{translations} of $\Ti{K}$ and $\Rm{id}_{W_{\Delta,\Rm{sm}}}$.
\end{define}

\begin{prop}
It holds that 
\begin{align*}
T(\Ti{K})_{p}=\Ti{K}\cap W_{\Rm{sm}}(\Delta).
\end{align*}
In particular, $T(\Ti{K})_{p}$ is a subgroup of $\Ti{K}$.
\end{prop}

\begin{proof}
Since $W_{\Rm{sm}}(\Delta)$ is a normal subgroup of $\Rm{Aut} (W_{\Rm{sm}}/\Delta)$,
we deduce that $\Ti{K}\cap W_{\Delta,\Rm{sm}}(\Delta)$ is a normal subgroup of $\Ti{K}$.
We have an exact commutative diagram
\[\xymatrix{
1 \ar[r] \ar@{=}[d]&
\Ti{K}\cap W_{\Rm{sm}}(\Delta)  \ar[r] \ar@{^{(}->}[d]&
 \Ti{K} \ar[r] \ar@{^{(}->}[d]&
 \Ti{K}/(\Ti{K}\cap W_{\Rm{sm}}(\Delta)) \ar[r] \ar@{^{(}->}[d]&
 1 \ar@{=}[d]\\
1 \ar[r] &
 W_{\Rm{sm}} (\Delta)\ar[r]&
 \Rm{Aut}(W_{\Delta,\Rm{sm}}/\Delta) \ar[r]&
 \Rm{Aut} (W_{\Delta,\Rm{sm}}/\Delta,O_{\Delta}) \ar[r]&
 1.\\
}\]
We note that the injectivity of $\Ti{K}/(\Ti{K} \cap W_{\Delta,\Rm{sm}}(\Delta)) \to \Rm{Aut} (W_{\Delta,\Rm{sm}}/\Delta,O_{\Delta})$ \Red{is shown} by a simple \Blue{diagram} chasing.
Therefore we have the following exact commutative diagram  \[\xymatrix{
1 \ar[r] \ar@{=}[d]&
\Ti{K}\cap W_{\Rm{sm}}(\Delta)  \ar[r] \ar@{^{(}->}[d]&
 \Ti{K} \ar[r] \ar@{^{(}->}[d]&
 \Ti{K}/(\Ti{K}\cap W_{\Rm{sm}}(\Delta)) \ar[r] \ar@{^{(}->}[d]&
 1 \ar@{=}[d]\\
1 \ar[r] &
 (\Gamma_{q,\Rm{sm}},O_{q}) \ar[r]&
 \Rm{Aut}(\Gamma_{q,\Rm{sm}}) \ar[r]&
 \Rm{Aut} (\Gamma_{q,\Rm{sm}}, O_{q}) \ar[r]&
 1\\
}\]
for an arbitrary fiber $\Gamma_{q}$ ($q \in \Delta$) of $\varphi_{\Delta}$.
Since $\Ti{K}/(\Ti{K}\cap W_{\Delta,\Rm{sm}}(\Delta)) \to \Rm{Aut} (\Gamma_{q,\Rm{sm}}, O_{q}) $ is injective,
we deduce that 
\begin{align*} 
\Ti{K}\cap (\Gamma_{q},O_{q}) \subset \Ti{K}\cap W_{\Rm{sm}}(\Delta)
\end{align*}
for \Blue{an arbitrary} smooth fiber $\Gamma_{q}$.

Take a non-trivial arbitrary automorphism $\kappa \in T(\Ti{K})_{p}$.
Since $\kappa$ is \Blue{a translation} of $\Ti{K}$, 
$\kappa|_{\Gamma_{q}}$ has no fixed point for general $q \in \Delta$.
It implies that $\kappa|_{\Gamma_{q}} \in \Ti{K}\cap (\Gamma_{q},O_{q})$.
By $\Ti{K}\cap (\Gamma_{q},O_{q}) \subset \Ti{K}\cap W_{\Rm{sm}}(\Delta)$, 
we have $\kappa \in\Ti{K}\cap W_{\Rm{sm}}(\Delta)$.
Therefore we obtain $T(\Ti{K})_{p} \subset \Ti{K}\cap W_{\Rm{sm}}(\Delta)$.
The converse $\Ti{K}\cap W_{\Rm{sm}}(\Delta) \subset  T(\Ti{K})_{p}$ is trivial.
\end{proof}

\begin{cor}
\label{groupstr}
There exists an exact commutative diagram 
\[\xymatrix{
1 \ar[r] \ar@{=}[d]&
T(\Ti{K})_{p}  \ar[r] \ar@{^{(}->}[d]&
 \Ti{K} \ar[r] \ar@{^{(}->}[d]&
 \Ti{K}/T(\Ti{K})_{p} \ar[r] \ar@{^{(}->}[d]&
 1 \ar@{=}[d]\\
1 \ar[r] &
 (\Gamma_{q,\Rm{sm}},O_{q}) \ar[r]&
 \Rm{Aut}(\Gamma_{q,\Rm{sm}}) \ar[r]&
 \Rm{Aut} (\Gamma_{q,\Rm{sm}}, O_{q}) \ar[r]&
 1\\
}\]
for an arbitrary fiber $\Gamma_{q}$ $(q \in \Delta)$ of $\varphi_{\Delta}$.
\end{cor}
The following two Corollaries are mentioned in \cite{Mir}. 
\begin{cor}
\label{nonss}
If $\varphi:W \to B$ has a singular fiber $\Gamma_{p}$ $(p \in B)$ except for type $l\Rm{I}_{c}$.
Then it holds that $\sharp T(\Ti{K})_{p}\leq 4$.
\end{cor}

\begin{proof}
From Theorem~\ref{kodaira} and Corollary~\ref{groupstr}, we have an injective homomorphism 
\begin{align*}
T(\Ti{K})_{p} \hookrightarrow \Bb{G}_{a}\times G(\Gamma_{p,\Rm{sm}}). 
\end{align*}
Since $\Bb{G}_{a}$ is a torsion free group, we have $T(\Ti{K})_{p} \hookrightarrow  G(\Gamma_{p,\Rm{sm}})$.
By the assumption that $\Gamma_{p}$ is not of type $l\Rm{I}_{c}$, we have $\sharp G(\Gamma_{p,\Rm{sm}})\leq 4$.
Thus, it holds that $\sharp T(\Ti{K})_{p}\leq 4$.
\end{proof}

\begin{cor}
\label{stabfree}
Let $\Gamma_{p}$ be a singular fiber of type $\Rm{I}_{c}$.
Then the only points of $\Gamma_{p}$ with non-trivial stabilizer subgroup under the action of $T(\Ti{K})_{p}$ are \Blue{nodes} of 
$\Gamma_{p}$. 
\end{cor}

\begin{cor}
\label{K/T(K)}
Assume $\varphi:W \to B$ is not isotrivial.
Then it holds that $\sharp(\Ti{K}/T(\Ti{K})_{p})\leq 2$ for any $p \in B$.
In particular, it holds that $\sharp \Ti{K}\leq 2\sharp T(\Ti{K})_{p}$.
\end{cor}

\begin{proof}
Let $\varphi_{\Delta}:W_{\Delta} \to \Delta$ be a restriction of $\varphi:W \to B$ to $p \in \Delta$.
Since $\varphi:W \to B$ is not isotrivial,
there exists a smooth fiber $\Gamma_{q}$ of $\varphi|_{\Delta}$
such that $\sharp (\Rm{Aut}(\Gamma_{q},O_{q}))=2$.
From Corollary~\ref{groupstr}, we have the injective homomorphism 
$\Ti{K}/T(\Ti{K})_{p} \hookrightarrow \Rm{Aut} (\Gamma_{q},O_{q})$.   
\end{proof}

\smallskip

We suppose that $\Gamma_{p}$ is the multiple fiber of type $l\Rm{I}_{c}$.
Let $\pi:\Delta^{\dagger} \to \Delta$ be a $l$-th root map branched at $p$.
Then there exists a commutative diagram 
\begin{equation}
\label{diagram}
\xymatrix{
W_{\Delta^{\dagger}} \ar[rd]  \ar@/_1.5pc/[ddr]_{\varphi_{\Delta^{\dagger}}} \ar@/^1.5pc/[rrd]^{\Pi}&  & \\
  &W_{\Delta}\times_{\Delta}\Delta^{\dagger} \ar[r] \ar[d]& W_{\Delta} \ar[d]^{\varphi_{\Delta}} \\
& \Delta^{\dagger}  \ar[r]_{\pi} &\Delta
}
\end{equation}
where $W_{\Delta^{\dagger}}$ is the normalization of $W_{\Delta}\times_{\Delta}\Delta^{\dagger}$.
We note that $W_{\Delta^{\dagger}}$ is non-singular.
Since $\pi^{-1}(p)$ \Red{consists of} only one point,
denote this point by $p^{\dagger}$.
Let $\Gamma_{p^{\dagger}}$ be a fiber of $\varphi_{\Delta^{\dagger}}$ over ${p^{\dagger}}$.
We note that $\Gamma_{p^{\dagger}}$ is a singular fiber of type $\Rm{I}_{lc}$.
Put $R_{\dagger}:=\Pi^{\ast} R$.
We note that $\Pi:W_{\Delta^{\dagger}} \to W_{\Delta}$ is an unramified covering of degree $l$.

\begin{prop}
There exists an injective homomorphism $\Rm{Aut}(W_{\Delta}/\Delta)\to \Rm{Aut} (W_{\Delta^{\dagger}}/{\Delta^{\dagger}})$.
\end{prop}

\begin{proof}
An automorphism $\kappa:W_{\Delta} \to W_{\Delta} \in \Rm{Aut}(W_{\Delta}/\Delta)$
induces the automorphism 
\begin{align*}
\kappa \times _{\Delta} \Delta^{\dagger}:W_{\Delta}\times_{\Delta}\Delta^{\dagger} \to W_{\Delta}\times_{\Delta}\Delta^{\dagger} 
\end{align*}
by the universality of the fiber product.
Hence there exists the injective homomorphism 
$\Rm{Aut}(W_{\Delta}/\Delta)\to \Rm{Aut} (W_{\Delta}\times_{\Delta}\Delta^{\dagger} /{\Delta^{\dagger}})$.
Furthermore, the universality of normalization induces 
the injective homomorphism
\begin{align*}
\Rm{Aut} (W_{\Delta}\times_{\Delta}\Delta^{\dagger} /{\Delta^{\dagger}})\hookrightarrow\Rm{Aut}(W_{\Delta^{\dagger}}/\Delta^{\dagger}).
\end{align*}
Thus, we get the desired homomorphism.
\end{proof}

Hence we can regard $\Ti{K}$ as \Blue{a subgroup} of $\Rm{Aut}(W_{\Delta^{\dagger}}/{\Delta^{\dagger}})$.
We denote by $\kappa^{\dagger} \in \Rm{Aut}(W_{\Delta^{\dagger}}/{\Delta^{\dagger}})$ the action 
which corresponds to $\kappa \in \Rm{Aut}(W_{\Delta}/\Delta)$ and 
let $\Ti{K}^{\dagger}:=\{ \kappa^{\dagger} \mid \kappa \in \Ti{K}\}$.

\begin{prop}
\label{lifting}
\Blue{Let $z^{\dagger} \in \Gamma_{p^{\dagger}}$ be a point and put $z=\Pi(z^{\dagger}) \in \Gamma_{p}$}.
\Blue{Let} $\kappa \in \Rm{Stab}_{\Ti{K}}(z)$.
Then $\kappa^{\dagger} \in \Rm{Stab}_{\Ti{K}^{\dagger}}(z^{\dagger})$.
In particular, it holds $\sharp(\Ti{K}\cdot z)=\sharp(\Ti{K}^{\dagger}\cdot z^{\dagger})$ 
\end{prop}

\begin{proof}
(1) The case that $\Gamma_{p^{\dagger}}$ is smooth at $z^{\dagger}$.

Then $\Gamma_{p,\Rm{red}}$ is smooth at $z$.
We take analytic local neighborhoods $(U;(x_{1},x_{2}))$
of $z$ on $W_{\Delta}$, 
$(\Delta;y)$ of $p$ and $(\Delta^{\dagger};x_{3})$ of $p^{\dagger}$ such that $\varphi^{\ast}y=x_{1}^{l}$ and $\pi^{\ast}y=x_{3}^{l}$.
Then $U\times_{\Delta}\Delta^{\dagger}$ is defined by $x_{1}^{l}-x_{3}^{l}=0$ in $U\times\Delta^{\dagger}$.
Put $H_{i}:=\{(x_{1},x_{2},\zeta_{l}^{i}x_{1})\mid(x_{1},x_{2}) \in U\}$ for $i=0,\cdots,l-1$,
where $\zeta_{l}$ is a primitive $l$-th root of unity.
Then we have
\[
U\times_{\Delta}\Delta^{\dagger}=\bigcup_{i=0}^{l-1} H_{i} \subset U\times \Delta .
\]
Thus, the natural projection $\sqcup_{i=0}^{l-1}H_{i}\to U\times_{\Delta}\Delta^{\dagger}$ is nothing 
more than the normalization of $U\times_{\Delta}\Delta^{\dagger}$.
Let $\Rm{id} \neq \kappa \in \Rm{Stab}_{\Ti{K}}(z)$.
Replacing $U$ by a smaller neighborhood of $z$ if necessary, we may assume that $\kappa (U) = U$.
We write $\kappa(x_{1},x_{2})=(\kappa_{1}(x_{1,x_{2}}),\kappa_{2}(x_{1},x_{2}))$ on U.
Then the induced automorphism 
$\kappa\times_{\Delta}\Delta^{\dagger}:U\times_{\Delta}\Delta^{\dagger} \to U\times_{\Delta}\Delta^{\dagger}$
by $\kappa$ is written by $(\kappa\times_{\Delta}\Delta^{\dagger})(x_{1},x_{2},x_{3}) = (\kappa_{1}(x_{1},x_{2}),\kappa_{2}(x_{1},x_{2}),x_{3})$.
It holds that $(\kappa\times_{\Delta}\Delta^{\dagger})(H_{i})=H_{i}$ for $i=0,\cdots,l-1$.
Furthermore, we have a commutative diagram
\[\xymatrix{
H_{i} \ar[d]_{\Pi|_{H_{i}}}   \ar[rr]^{(\kappa\times_{\Delta}\Delta^{\dagger})}                  & &H_{i} \ar[d]^{\Pi|_{H_{i}}}\\
U    \ar[rr]^{\kappa}    &         & U\\
}\]
 for $i=0,\cdots,l-1$.
Thus, we have $\kappa^{\dagger} \in \Rm{Stab}_{\Ti{K}}(z^{\dagger})$.

(2) The case that $\Gamma_{p^{\dagger}}$ has a node at $z^{\dagger}$.

Then $\Gamma_{p.\Rm{red}}$ has a node at $z$.
We take analytic local neighborhoods $(U;(x_{1},x_{2}))$
of $z$ on $W_{\Delta}$, 
$(\Delta;y)$ of $p$ and $(\Delta^{\dagger};x_{3})$ of $p^{\dagger}$ such that $\varphi^{\ast}y=(x_{1}x_{2})^{l}$ and $\pi^{\ast}y=x_{3}^{l}$.
Then $U\times_{\Delta}\Delta^{\dagger}$ is defined by $(x_{1}x_{2})^{l}-x_{3}^{l}=0$ in $U\times\Delta^{\dagger}$.
Put $H_{i}:=\{(x_{1},x_{2},\zeta_{l}^{i}x_{1}x_{2})\mid(x_{1},x_{2}) \in U\}$ for $i=0,\cdots,l-1$,
where $\zeta_{l}$ is a primitive $l$-th root of unity.
Then we have
\[
U\times_{\Delta}\Delta^{\dagger}=\bigcup_{i=0}^{l-1} H_{i} \subset U\times \Delta .
\]
Thus, the natural projection $\sqcup_{i=0}^{l-1}H_{i}\to U\times_{\Delta}\Delta^{\dagger}$ is the normalization of $U\times_{\Delta}\Delta^{\dagger}$.
We can show  $\kappa^{\dagger} \in \Rm{Stab}_{\Ti{K}}(z^{\dagger})$ similarly to (1).
\end{proof}
We can define $T(\Ti{K}^{\dagger})_{p^{\dagger}}$ for 
$\Ti{K}^{\dagger} \subset \Rm{Aut} (W_{\Delta^{\dagger}}/{\Delta^{\dagger}})$
since $\Gamma_{p^{\dagger}}$ is not multiple.
Hence we define 
\[
T(\Ti{K})_{p}:=\{\kappa \in \Ti{K}\mid \kappa^{\dagger} \in T(\Ti{K}^{\dagger})_{p^{\dagger}} \}
\]
for a multiple fiber $\Gamma_{p}$.
We note that $T(\Ti{K})_{p} \cong T(\Ti{K}^{\dagger})_{p^{\dagger}} $
 and $\Ti{K} \cong \Ti{K}^{\dagger}$.

\section{Estimation of the order of $\sharp \Ti{K}$}

Let $f: S\to B$ be a primitive cyclic covering fibration of type $(g,1,n)$ with \Blue{$r=\frac{2(g-1)}{n-1} \geq 4n$}.
\begin{define}
Let $z \in R_{h}$ be a (smooth) ramification point of $\varphi|_{R_{h}}$.
We call $z$ a {\it good ramification point} when the ramification index of $\varphi|_{R_{h}}$ at 
$z$ is greater than the multiplicity of the fiber of $\varphi$ passing through $z$.
\end{define}

The \Blue{goal} in this section is to show Proposition~\ref{estimateK2}, \ref{estimateK1} and \ref{estimateK1.5}.

\begin{define}
\Blue{
Put
\begin{align*}
\delta:=\Rm{min}\{ \sharp \Rm{Aut}(\Gamma_{p},O_{p})|\;
\forall p \in \Delta \textrm{ with }\Gamma_{p} \textrm{ is smooth} \},
\end{align*}
where $\Rm{Aut}(\Gamma_{p},O_{p})$ is the same notation defined in (\ref{Fixauto}).}
We note that $\delta=2,4$ or $6$.
\end{define}

\begin{prop}
\label{estimateK2}
Let $f:S \to B $ be a primitive cyclic covering fibration of \Blue{$(g,1,n)$} with $r \geq 4n$.
\Red{Let $\Gamma_{p}$ be} of type $l\Rm{I}_{c}$.
Assume $R$ is smooth locally around $\Gamma_{p}$ and 
has no good ramification points on $\Gamma_{p}$.
Then it holds that 
\begin{align*}
 K_{f}^{2}(\Gamma_{p}) & \geq \frac{(n^{2}-1)}{12n\delta}\sharp \Ti{K} .
\end{align*}
\end{prop}
\begin{proof}
Since $R$ is smooth locally around $\Gamma_{p}$,
we have $j_{0,a}'(\Gamma_{p})=0$ for any $a \geq 1$ if $n \geq 3$.
Similarly, we have $j_{0,a}'(\Gamma_{p})=0$ for $a \geq 2$ if $n=2$.
Hence we have 
\begin{align*}
K_{f}^2(\Gamma_{p}) = \frac {(n-1)^2}{n}\alpha_{0}^{+}(\Gamma_{p})
+\frac{n^2-1}{n}r\left(\chi_{f}(\Gamma_{p})+1-\frac{1}{l}\right)
\end{align*}
from Lemma~\ref{Enolem}.
Since the ramification index of any ramification point of 
$\varphi|_{R_{h}}$ over $p$ is $l$,
we have 
\[
\alpha_{0}^{+}(\Gamma_{p})=r(1-\frac{1}{l}).
\]
By \Blue{Corollary}~\ref{groupstr}, we have $\delta r \geq \sharp \Ti{K}$.
Thus, we have 
\begin{align*}
K_{f}(\Gamma_{p}) &\geq \left(\frac{2(n-1)}{\delta}\left(1-\frac{1}{l} \right)+\frac{n^2-1}{\delta n}\frac{c}{12}\right)\sharp \Ti{K}\\
&\geq \frac{(n^{2}-1)}{12n\delta}\sharp \Ti{K}.
\end{align*}
\end{proof}

\begin{prop}
\label{estimateK1}
Let $f:S \to B $ be a primitive cyclic covering fibration of \Blue{$(g,1,n)$} with $r \geq \Rm{max}\{60 + \frac{12}{n^2-1}-\frac{96}{n+1}, 7n\}$.
Assume either $\Gamma_{p}$ is not of type $l\Rm{I}_{c}$  
or $R$ has a singular point on $\Gamma_{p}$.
Then it holds that
\begin{align*}
 2 n \delta  K_{f}^{2}(\Gamma_{p}) & \geq  \frac{1}{3}(n-1)(5n-4)\sharp \Ti{K}.
\end{align*}
\end{prop}

\begin{prop}
\label{estimateK1.5}
Let $f:S \to B $ be a primitive cyclic covering fibration of \Blue{$(g,1,n)$} with $r \geq 4n$.
Let $\Gamma_{p}$ is of type $l\Rm{I}_{c}$.
Assume $R$ is smooth locally around $\Gamma_{p}$ and 
has a good ramification point on $\Gamma_{p}$.
Then it holds that
\begin{align*}
 2 n \delta  K_{f}^{2}(\Gamma_{p}) & \geq  (n-1)^{2}\sharp \Ti{K}.
\end{align*}
\end{prop}

\subsection{The proof of Proposition~\ref{estimateK1}}

\begin{lem}
Let $f:S \to B $ be a primitive cyclic covering fibration of \Blue{$(g,1,n)$} with $r \geq 4n$.
Assume $\Gamma_{p}$ is a singular fiber except for type $l\Rm{I}_{c}$.
Then it holds that 
\begin{align*}
2n\delta K_{f}^{2}(\Gamma_{p}) &\geq \frac{n^2-1}{12}\left(r-12\frac{n-1}{n+1} \right) \sharp \Ti{K}. 
\end{align*}
If $r \geq 7n$, it holds
\begin{align*}
2n\delta K_{f}^{2}(\Gamma_{p}) &\geq  \frac{1}{3}(n-1)(5n-4) \sharp \Ti{K}. 
\end{align*}
\end{lem}

\begin{proof}

Let $\Gamma_{p}$ be a singular \Blue{fiber except} for type $l\Rm{I}_{c}$.
We note that $e(\Gamma_{p}) \geq \sharp\Gamma_{p}+1$ in this case.
So we have $j_{0,\bullet}'(\Gamma_{p}) \leq \sharp \Gamma_{p}$ and $\chi_{\varphi}(\Gamma_{p})\geq (\sharp\Gamma_{p}+1)/12$. 
We have 
\begin{align*}
K_{f}^{2} (\Gamma_{p}) &\geq -2\frac{(n-1)^2}{n}j_{0,\bullet}'(\Gamma_{p})+
j_{0,1}'(\Gamma_{p})+\frac{n^{2}-1}{n}r\chi_{\varphi}(\Gamma_{p})
\end{align*}
from Proposition~\ref{localization} and Lemma~\ref{lowerbdK1}.
Hence we have
\begin{align*}
K_{f}^{2} (\Gamma_{p}) &\geq -2\frac{(n-1)^2}{n}j_{0,\bullet}'(\Gamma_{p})+
j_{0,1}'(\Gamma_{p})+\frac{n^{2}-1}{12n}r(\sharp\Gamma_{p}+1)\\
&=\left(\frac{n^{2}-1}{12n}r-\frac{2(n-1)^2}{n} \right)\sharp \Gamma_{p} + \frac{n^{2}-1}{12n}r \\
&\geq \frac{n^{2}-1}{6n}r-\frac{2(n-1)^2}{n}\;(\text{by}\;\sharp \Gamma_{p} \geq 1)    \\
&=\frac{n^{2}-1}{6n}\left(r - 12 \frac{n-1}{n+1} \right).
\end{align*}
Thus, 
it is sufficient to show 
\begin{align*}
2n\delta \frac{n^{2}-1}{6n}\left(r - 12 \frac{n-1}{n+1} \right)
- \frac{n^{2}-1}{12}\left(r - 12 \frac{n-1}{n+1} \right) \sharp \Ti{K} \geq 0.
\end{align*}
Since $\Gamma_{p}$ is not of type $l\Rm{I}_{c}$,
we have $\sharp \Ti{K} \leq 4\delta $ from Corollary~\ref{nonss}.
Hence we get the desired inequality.
\end{proof}


We assume $\varphi:W \to B$ has at most singular fibers of type $l\Rm{I}_{c}$ in what follows in this subsection.
Let $z$ be a singular point of $R$ and \Red{$z \in \Gamma_{p}$ a fiber }of type $l\Rm{I}_{c}$ where $l$
and $c$ are non-negative integers.
Recall the diagram (\ref{diagram}).
\[\xymatrix{
W_{\Delta^{\dagger}} \ar[rd]  \ar@/_1.5pc/[ddr]_{\varphi_{\Delta^{\dagger}}} \ar@/^1.5pc/[rrd]^{\Pi}&  & \\
  &W_{\Delta}\times_{\Delta}\Delta^{\dagger} \ar[r] \ar[d]& W_{\Delta} \ar[d]^{\varphi_{\Delta}} \\
& \Delta^{\dagger}  \ar[r]_{\pi} &\Delta
}\]
We recall that $\pi^{-1}(p)=\{p^{\dagger}\}$ and $R_{\dagger}=\Pi^{\ast}(R|_{W_{\Delta}})$.
Replacing $\Delta$ by a smaller neighborhood of $p$ if necessary,
we may assume that $\varphi_{\Delta^{\dagger}}$ has only one singular fiber $\Gamma_{p^{\dagger}}$ and \Red{$R_{\dagger}$ has a singular point} only on $\Gamma_{p^{\dagger}}$.

Since $\Pi:W_{\Delta^{\dagger}} \to W_{\Delta}$ is an unramified covering,
each singular point of $R_{\dagger}$ 
is \Blue{analytically equivalent} to the corresponding singular point of $R|_{W_{\Delta}}$.
Hence we can consider a formal canonical resolution of $R_{\dagger}$ as follows.

\begin{define}
Let $z^{\dagger}_{1}$ be a singular point of $R_{\dagger}$ on $\Gamma_{p^{\dagger}}$ and 
\Red{$(\psi_{\Delta^{\dagger}})_{1}:W_{\Delta^{\dagger},1} \to W_{\Delta^{\dagger}}$ the blowing-up} at $z^{\dagger}_1$.
We put 
\begin{align*}
R_{\dagger,1}:=(\psi_{\Delta^{\dagger}})^{\ast}R_{\dagger}-n\left[\frac{m_1}{n}\right]E_{1},
\end{align*}
where $E_1:=(\psi_{\Delta^{\dagger}})^{-1}_{1}(z^{\dagger}_1)$ 
and $m_1:=\Rm{mult}_{z^{\dagger}_1}(R_{\dagger})$.
We define $(W_{\Delta^{\dagger},i},R_{\dagger,i})$ for $i=2, \cdots, N^{\dagger}$ inductively
as follows.
Let $z^{\dagger}_{i}$ be a singular point of $R_{\dagger,i-1}$ and 
\Red{$(\psi_{\Delta^{\dagger}})_{i}:W_{\Delta^{\dagger},i} \to W_{\Delta^{\dagger},i-1}$ the 
blowing-up} at $z^{\dagger}_{i}$.
We put 
\begin{align*}
R_{\dagger,i}:=(\psi_{\Delta^{\dagger}})_{i}^{\ast}R_{\dagger,i-1}-n\left[\frac{m_i}{n}\right]E_{i},
\end{align*}
where $E_{i}:=(\psi_{\Delta^{\dagger}})^{-1}_{i}(z^{\dagger}_i)$ 
and $m_{i}=\Rm{mult}_{z^{\dagger}_{i}}(R_{\dagger,i-1})$.
We continue this process until $R_{\dagger,i}$ is smooth.
Since $\Pi$ is an unramified covering,
this process terminates.
Let $(W_{\Delta^{\dagger},N^{\dagger}},R_{\dagger,N^{\dagger}})$ 
be the model which the above process terminates 
and put $\Ti{\psi}_{\Delta^{\dagger}}:=(\psi_{\Delta^{\dagger}})_{1}\circ \cdots \circ (\psi_{\Delta^{\dagger}})_{N^{\dagger}}$.
We call $R_{\dagger,i}$ ($i=1, \cdots N^{\dagger}$) \Blue{the {\it branch locus}} on $W_{\Delta^{\dagger},i}$.
We rewrite $(W_{\Delta^{\dagger},N^{\dagger}},R_{\dagger,N^{\dagger}})$ as $(\Ti{W}_{\Delta^{\dagger}},\Ti{R}_{\dagger})$.
\end{define}

\begin{rmk}
\label{formalpcc}
The following hold.
\begin{itemize}
\item The multiplicity of an arbitrary singular point of $R_{\dagger,i}$
is in $n\Z$ or $n\Z+1$ for $i=1, \cdots , N^{\dagger}$. 
\item The self-intersection number of any vertical component of $\Ti{R}_{\dagger}$ which is contracted to a point by $\Ti{\psi}_{\Delta^{\dagger}}$ is divisible by $n$.
\end{itemize}
\end{rmk}

 Let $\Ha{\varphi}_{\Delta^{\dagger}}: \Ha{W}_{\Delta^{\dagger}} \to \Delta^{\dagger}$ be any intermediate elliptic fiber germ between $\Ti{W}_{\Delta^{\dagger}}$ and $W_{\Delta^{\dagger}}$, and regard $\Ti{\psi}_{\Delta^{\dagger}}:\Ti{W}_{\Delta^{\dagger}}\to W_{\Delta^{\dagger}}$ as the composite of 
the natural birational morphisms $\Ha{\psi}_{\Delta^{\dagger}}: \Ti{W}_{\Delta^{\dagger}}\to \Ha{W}_{\Delta^{\dagger}}$ and $\Ch{\psi}:\Ha{W}_{\Delta^{\dagger}}\to W_{\Delta^{\dagger}}$.
\[
  \xymatrix{
    \Ti{W}_{\Delta^{\dagger}} \ar[r]^{\Ha{\psi}_{\Delta^{\dagger}}} \ar[dr]_{\Ti{\varphi}_{\Delta^{\dagger}}} & \Ha{W}_{\Delta^{\dagger}} \ar[d]^{\Ha{\varphi}_{\Delta^{\dagger}}} \ar[r]^{\Ch{\psi}_{\Delta^{\dagger}}}& W_{\Delta^{\dagger}} \ar[ld]^{\varphi_{\Delta^{\dagger}}}\\
     & B &
}
\]
We put $\Ha{R}_{\dagger}=(\Ha{\psi}_{\Delta^{\dagger}})_{\ast}\Ti{R}_{\dagger}$.
The fiber of $\Ha{\varphi}_{\Delta^{\dagger}}$ over $p \in \Delta^{\dagger}$ will be denoted by $\Ha{\Gamma}_{p^{\dagger}}$.

For any $\Ha{z}^{\dagger} \in \Ha{\Gamma}_{p^{\dagger}}$, we introduce indices as Section~2.
\begin{align*}
\begin{aligned}
\alpha_0^{+} (\Gamma_{p^{\dagger}})_{\Ha{z}^{\dagger}} &:=
&  \begin{cases}
   \left( (\Ti{R}_{\dagger})_{h}, \Ti{\Gamma}_{p^{\dagger},\Ha{z}^{\dagger}}\right)-\sharp \left(\Rm{Supp}  (\Ti{R}_{\dagger})_{h} \cap \Rm{Supp}(\Ti{\Gamma}_{p^{\dagger},\Ha{z}^{\dagger}})\right) &({\rm if\;} \Ha{R}_{\dagger}{\rm\;is\;singular\;at\;} \Ha{z}^{\dagger} \;),\\
    \\
    \left({\rm The\;ramification\;index\;of\;} (\Ti{R}_{\dagger})_{h} \to \Delta^{\dagger}{\rm \;at\;}\Ha{z}\right)-1 &({\rm if\;} \Ha{R}_{\dagger}  {\rm\;is\;smooth\;at\;} \Ha{z}^{\dagger} \;),\\
    \\
   0 & ({\rm if\;} \Ha{z}^{\dagger} \not\in  \Ha{R}_{\dagger}),
  \end{cases}
\end{aligned}
\end{align*}
where $\Ti{\Gamma}_{p^{\dagger},\Ha{z}^{\dagger}}$ is the biggest subdivisor of $\Ti{\Gamma}_{p^{\dagger}}$ which is contracted to $\Ha{z}^{\dagger}$ by $\Ha{\psi}_{\Delta^{\dagger}}$.
\begin{itemize}
\item Let $\Ch{\alpha}_{k}(\Gamma_{p^{\dagger}})$
be the number of singular points of multiplicity either $kn$ or $kn+1$ among 
all singular points of \Red{the branch loci} appearing in $\Ch{\psi}_{\Delta^{\dagger}}$.
If $\Ch{\psi}_{\Delta^{\dagger}}=\Rm{id}_{W_{\Delta^{\dagger}}}$, we define $\Ch{\alpha}_{k}(\Gamma_{p^{\dagger}})=0$.
\item Let $\Ha{\alpha}_{k}(\Gamma_{p^{\dagger}})_{\Ha{z}^{\dagger}}$
be the number of singular points of multiplicity either $kn$ or $kn+1$ among 
all singular points of \Red{the branch loci} which are infinitely near to $\Ha{z}^{\dagger}$.
If $\Ch{\psi}_{\Delta^{\dagger}}=\Rm{id}_{W_{\Delta^{\dagger}}}$, we rewrite $\Ha{\alpha}_{k}(\Gamma_{p^{\dagger}})_{\Ha{z}^{\dagger}}$ as $\alpha_{k}(\Gamma_{p^{\dagger}})_{\Ha{z}^{\dagger}}$.
 \item Put $j_{0,a}^{'}(\Gamma_{p^{\dagger}}):= l j_{0,a}^{'}(\Gamma_{p})$ for $a \geq 1$. 
 Put $j_{0,\bullet}^{'}(\Gamma_{p^{\dagger}}):= \sum_{a\geq 1}j_{0,a}^{'}(\Gamma_{p^{\dagger}})$
 \item \Blue{Let $\Ch{j}_{0,a}(\Gamma_{p^{\dagger}})$ be the number of vertical irreducible  components of $\Ti{R}_{\dagger}$ that satisfy following two conditions.
 \begin{itemize}
 \item The self-intersection number is $-an$.
 \item It is a proper transform of an exceptional curve appearing from $\Ch{\psi}_{\Delta^{\dagger}}$.
 \end{itemize}
 }
 Put $\Ch{j}_{0,\bullet}(\Gamma_{p^{\dagger}}):= \sum_{a\geq 1}\Ch{j}_{0,a}(\Gamma_{p^{\dagger}})$.
If $\Ch{\psi}_{\Delta^{\dagger}}=\Rm{id}_{W_{\Delta^{\dagger}}}$, we define $\Ch{j}_{0,\bullet}(\Gamma_{p^{\dagger}})=0$.
\item Let $\Ha{j}_{0,a}(\Gamma_{p^{\dagger}})_{\Ha{z}^{\dagger}}$ be the number of irreducible curves with self-intersection number \Red{$-an$} 
 contained in $(\Ti{R}_{\dagger})_{v}(p^{\dagger})$ and
 contracted to $\Ha{z}^{\dagger}$ by $\Ha{\psi}_{\Delta^{\dagger}}$.
  Put $\Ha{j}_{0,\bullet}(\Gamma_{p^{\dagger}})_{\Ha{z}^{\dagger}}:= \sum_{a\geq 1}\Ha{j}_{0,a}(\Gamma_{p^{\dagger}})_{\Ha{z}^{\dagger}}$.
If $\Ch{\psi}_{\Delta^{\dagger}}=\Rm{id}_{W_{\Delta^{\dagger}}}$, we rewrite
$\Ha{j}_{0,a}(\Gamma_{p^{\dagger}})_{\Ha{z}^{\dagger}}$ as 
$j_{0,a}^{''}(\Gamma_{p^{\dagger}})_{\Ha{z}^{\dagger}}$.
  \item Put $\Ha{\alpha}_{0}(\Gamma_{p^{\dagger}})_{\Ha{z}^{\dagger}}:= \alpha_{0}^{+}(\Gamma_{p^{\dagger}})_{\Ha{z}^{\dagger}}-2\sum_{a \geq 2}\Ha{j}_{0,a}(\Gamma_{p^{\dagger}})_{\Ha{z}^{\dagger}}$.
If $\Ch{\psi}_{\Delta^{\dagger}}=\Rm{id}_{W_{\Delta^{\dagger}}}$, we rewrite $\Ha{\alpha}_{0}(\Gamma_{p^{\dagger}})_{\Ha{z}^{\dagger}}$ as $\alpha_{0}(\Gamma_{p^{\dagger}})_{\Ha{z}^{\dagger}}$.  
\item 
$K_{f}^2 (\Gamma_{p^{\dagger}})_{\Ha{z}^{\dagger}} :=  \sum_{k\geq 1}\left((n^{2}-1)k-n \right)
\Ha{\alpha}_{k}(\Gamma_{p^{\dagger}})_{\Ha{z}^{\dagger}} +\frac{(n-1)^{2}}{n}(\Ha{\alpha}_{0}(\Gamma_{p^{\dagger}})_{\Ha{z}^{\dagger}}  -2\Ha{j}_{0,1}(\Gamma_{p})_{\Ha{z}^{\dagger}})+ \Ha{j}_{0,1}(\Gamma_{p^{\dagger}})_{\Ha{z}^{\dagger}}$.
\item $\chi_{\varphi_{\Delta^{\dagger}}}(\Gamma_{p^{\dagger}}):= l \chi_{\varphi_{\Delta}}(\Gamma_{p})$.
\end{itemize}
Using indices of the case $\Ch{\psi}_{\Delta^{\dagger}}=\Rm{id}_{W_{\Delta^{\dagger}}}$, we put 
\begin{align}
 K_{f}^2 (\Gamma_{p^{\dagger}}) :=& \sum_{z^{\dagger} \in \Gamma_{p^{\dagger}}}  
 K_{f}^2 (\Gamma_{p^{\dagger}})_{z^{\dagger}}
 -2\frac{(n-1)^2}{n}j_{0,\bullet}'(\Gamma_{p^{\dagger}})
 +\frac{n^{2}-1}{n}r\chi_{\varphi_{\Delta^{\dagger}}}(\Gamma_{p^{\dagger}})  + j_{0,1}'(\Gamma_{p^{\dagger}}).
 \end{align}

\begin{prop}
\label{localbc}
Let the notation and the assumption \Blue{be as} above.
Then it holds that
\begin{align*}
K_{f}^2 (\Gamma_{p^{\dagger}}) = l K_{f}^2 (\Gamma_{p}) -2(n-1)r(l-1).
\end{align*}
In particular, we have 
\begin{align*}
 l K_{f}^2 (\Gamma_{p}) \geq K_{f}^2 (\Gamma_{p^{\dagger}}).
\end{align*}
\end{prop}
\begin{proof}
By simple calculations,
we have
\begin{align}
\label{K_bc}
\sum_{z^{\dagger}\in\Gamma_{p^{\dagger}}}K_{f}^2 (\Gamma_{p^{\dagger}})_{z^{\dagger}}
= l \sum_{z\in\Gamma_{p}}K_{f}^2 (\Gamma_{p})_{z} -\frac{(n-1)^2}{n}(l-1)r.
\end{align}
Applying Proposition~\ref{localization} for $\Ch{\psi}_{\Delta}=\Rm{id}_{W_{\Delta}}$, we have
\begin{align*}
 l \sum_{z\in \Gamma_{p}} K_{f}^2 (\Gamma_{p})_{z}  
 = &l K_{f}^2 (\Gamma_{p})
 +2\frac{(n-1)^{2}}{n}l j'_{0,\bullet}(\Gamma_{p})  - l j'_{0,1}(\Gamma_{p}) 
\\
 &-\frac{n^{2}-1}{n}rl\chi_{\varphi}(\Gamma_{p}) -\frac{n^{2}-1}{n}r(l-1).
 \end{align*}
Substituting (\ref{K_bc}) to the above equation,
we have 
\begin{align}
\label{K_loc}
 \sum_{z^{\dagger}\in\Gamma_{p^{\dagger}}}K_{f}^2 (\Gamma_{p^{\dagger}}) 
 = &l K_{f}^2 (\Gamma_{p})
 +2\frac{(n-1)^{2}}{n}l j'_{0,\bullet}(\Gamma_{p})  - l j'_{0,1}(\Gamma_{p}) 
\\
\nonumber
 -&\frac{n^{2}-1}{n}rl\chi_{\varphi}(\Gamma_{p}) -2r(l-1)(n-1)
  \end{align}
 On the other hand, we have 
 \begin{align*}
 K_{f}^2 (\Gamma_{p^{\dagger}}) = &\sum_{z^{\dagger} \in \Gamma_{p^{\dagger}}} K_{f}^2 (\Gamma_{p^{\dagger}})_{z^{\dagger}}
 -\frac{2(n-1)^2}{n}j'_{0,\bullet}(\Gamma_{p^{\dagger}})
 +\frac{n^{2}-1}{n}r\chi_{\varphi_{\Delta^{\dagger}}}(\Gamma_{p^{\dagger}})  + j'_{0,1}(\Gamma_{p^{\dagger}}).
 \end{align*}
Substituting (\ref{K_loc}) to the above equation,
we get  
\begin{align*}
K_{f}^2 (\Gamma_{p^{\dagger}}) = l K_{f}^2 (\Gamma_{p}) -2(n-1)r(l-1).
\end{align*}
\end{proof}

\begin{prop}[cf.~Proposition~\ref{localization}]
\label{localization'}
Let $\Ha{\varphi}_{\Delta^{\dagger}}:\Ha{W}_{\Delta^{\dagger}} \to \Delta^{\dagger}$ be any intermediate elliptic fiber germ between $\Ti{W}_{\Delta^{\dagger}}$ and $W_{\Delta^{\dagger}}$.
\[
  \xymatrix{
    \Ti{W}_{\Delta^{\dagger}} \ar[r]^{\Ha{\psi}_{\Delta^{\dagger}}} \ar[dr]_{\Ti{\varphi}_{\Delta^{\dagger}}} & \Ha{W}_{\Delta^{\dagger}} \ar[d]^{\Ha{\varphi}_{\Delta^{\dagger}}} \ar[r]^{\Ch{\psi}_{\Delta^{\dagger}}}& W_{\Delta^{\dagger}} \ar[ld]^{\varphi_{\Delta^{\dagger}}}\\
     & B &
}
\]
Then it holds that 
\begin{align*}
 K_{f}^2 (\Gamma_{p^{\dagger}}) = &\sum_{\Ha{z}^{\dagger}\in \Ha{\Gamma}_{p^{\dagger}}} K_{f}^2 (\Gamma_{p^{\dagger}})_{\Ha{z}^{\dagger}} 
 + \sum_{k\geq 1}\left((n^{2}-1)k-n \right)\Ch{\alpha}_{k} (\Gamma_{p^{\dagger}}) \\
 &-2\frac{(n-1)^{2}}{n}\left(\Ch{j}_{0,\bullet}(\Gamma_{p^{\dagger}})+j_{0,\bullet}'(\Gamma_{p^{\dagger}})\right)
 +\frac{n^{2}-1}{n}r\chi_{\varphi_{\Delta^{\dagger}}}(\Gamma_{p^{\dagger}})  + \Ch{j}_{0,1}(\Gamma_{p^{\dagger}})+j_{0,1}'(\Gamma_{p^{\dagger}}).
 \end{align*}
 \end{prop}
\begin{proof}
We can check it by simple calculations. 
\end{proof}
We \Blue{can show the} following by a similar argument as in Lemma~\ref{lowerbdK1}.
\begin{lem}[cf.~Lemma~\ref{lowerbdK1}]
\label{lowerbdK1'}
It holds that 
 \begin{align*}
 K_{f}^2(\Gamma_{p^{\dagger}})_{\Ha{z}^{\dagger}} \geq \sum_{k \geq 1} \left( (n^2-1)k-n-\frac{(n-1)^{2}}{n}B_{n}\right)
 \Ha{\alpha}_{k}(\Gamma_{p^{\dagger}})_{\Ha{z}^{\dagger}}.
 \end{align*}
 In particular, $K_{f}^2(\Gamma_{p^{\dagger}})_{\Ha{z}^{\dagger}}$ is non-negative.
 \end{lem}
 
\begin{lem}[cf.~Lemma~\ref{6/10,2}]
\label{6/15,1}
 It holds that 
 \begin{align*}
 K_{f}^2 (\Gamma_{p^{\dagger}}) \geq &\sum_{\Ha{z}^{\dagger}\in \Ha{\Gamma}_{p^{\dagger}}} 
  \frac{(n-1)^{2}}{n}\alpha^{+}_{0}(\Gamma_{p^{\dagger}})_{\Ha{z}^{\dagger}}
 +\sum_{\Ha{z}^{\dagger}\in \Ha{\Gamma}_{p^{\dagger}}}
 \sum_{k\geq 1}\left((n^{2}-1)k-n-2\frac{(n-1)^2}{n} \right)\Ha{\alpha}_{k} (\Gamma_{p^{\dagger}}) _{\Ha{z}^{\dagger}}\\
  &+ \sum_{k\geq 1}\left((n^{2}-1)k-n \right)\Ch{\alpha}_{k} (\Gamma_{p^{\dagger}}) 
 -2\frac{(n-1)^{2}}{n}\left(\Ch{j}_{0,\bullet}(\Gamma_{p^{\dagger}}) +j'_{0,\bullet}(\Gamma_{p^{\dagger}}) \right)\\
 &+\frac{n^{2}-1}{n}r\chi_{\varphi_{\Delta^{\dagger}}}(\Gamma_{p^{\dagger}}) + \Ch{j}_{0,1}(\Gamma_{p^{\dagger}})+j_{0,1}'(\Gamma_{p^{\dagger}}).
 \end{align*}
\end{lem}
\begin{proof}
We can check it by simple calculations. 
\end{proof}

\begin{lem}
\label{lifting2}
Let $z^{\dagger} \in R_{\dagger}$ be a singular point of $R_{\dagger}$ and let $\Pi(z^{\dagger})=z$. 
Then there \Red{exist} $l\cdot\sharp(T(\Ti{K}^{\dagger})_{p}\cdot z^{\dagger})$ singular points of $R_{\dagger}$ which are \Blue{analytically equivalent} to $z \in R$.
\end{lem}
\begin{proof}
We note that $\sharp\Pi^{-1}(T(\Ti{K})_{p}\cdot z)=l\cdot \sharp(T(\Ti{K})_{p}\cdot z)$ 
since $\Pi:W^{\dagger} \to W$ is an unramified covering of degree $l$. 
There exists an analytic neighborhood $U$ of each $z^{\dagger} \in \Pi^{-1}(T(\Ti{K})_{p}\cdot z) $ 
such that $\Pi|_{U}:U \to \Pi(U)$ is a biholomorphic map.
Hence the singular point $z^{\dagger} \in R_{\dagger}$ is \Blue{analytically equivalent} to $z \in R$.
Thus, there exist $l\cdot\sharp(T(\Ti{K})_{p}\cdot z)$ singular points of $R_{\dagger}$ which are 
\Blue{analytically equivalent} to $z \in R$.
On the other hand, we have $\sharp(T(\Ti{K})_{p}\cdot z)=\sharp(T(\Ti{K}^{\dagger})_{p}\cdot z^{\dagger})$
 from Proposition~\ref{lifting}.
\end{proof}

Let $(R_{\dagger})_{h}$ be the local analytic curve consisting of $\varphi_{\Delta^{\dagger}}$-horizontal local analytic branches of $R_{\dagger}$. 
Take a singular point $z^{\dagger} \in R_{\dagger}$ on $\Gamma_{p^{\dagger}}$. 
We introduce the following notations for $z^{\dagger}$.
\begin{itemize}
\item[(1)] We denote by $\Rm{L}(R_{\dagger},z^{\dagger})$ and $\Rm{L}((R_{\dagger})_{h},z^{\dagger})$ \Blue{the sets} of local analytic branches of $R_{\dagger}$ and $(R_{\dagger})_{h}$ at $z^{\dagger}$, respectively.
\item[(2)] Let $\Rm{L}_{\Rm{tr}}(R_{\dagger},z^{\dagger})$ and $\Rm{L}_{\Rm{tr}}((R_{\dagger})_{h},z^{\dagger})$ be \Blue{the subsets} of $\Rm{L}(R_{\dagger},z^{\dagger})$ and $\Rm{L}((R_{\dagger})_{h},z^{\dagger})$ which consist of local analytic branches that meet $\Gamma_{p^{\dagger}}$ transversally at $z^{\dagger}$, respectively.
\item[(3)] Let $\Rm{L}_{\Rm{ta}}(R_{\dagger},z^{\dagger})$ and $\Rm{L}_{\Rm{ta}}((R_{\dagger})_{h},z^{\dagger})$ be \Blue{the subsets} of $\Rm{L}(R_{\dagger},z^{\dagger})$ and $\Rm{L}((R_{\dagger})_{h},z^{\dagger})$ which consist of local analytic branches that are tangent  to 
$\Gamma_{p^{\dagger}}$ at $z^{\dagger}$, respectively.
\item[(4)] For a local analytic branch $\Gamma_{i}$ of $\Gamma_{p^{\dagger}}$ around $z^{\dagger}$,
let $\Rm{L}_{\Rm{ta}}(R_{\dagger},\Gamma_{i},z^{\dagger})$ and $\Rm{L}_{\Rm{ta}}((R_{\dagger})_{h},\Gamma_{i},z^{\dagger})$ be \Blue{the subsets} of $\Rm{L}_{\Rm{ta}}(R_{\dagger},z^{\dagger})$ and $\Rm{L}_{\Rm{ta}}((R_{\dagger})_{h},z^{\dagger})$ which consist of local analytic branches that are tangent to $\Gamma_{i}$ at $z^{\dagger}$, respectively.
\end{itemize}
\Blue{
If the local analytic branch $\Gamma_{i}$ of $\Gamma_{p^{\dagger}}$ is contained in $R_{\dagger}$,
we regard the local analytic branch $\Gamma_{i}$ as belonging to $\Rm{L}_{\Rm{ta}}(R_{\dagger},\Gamma_{i},z^{\dagger})$.
Hence we have $\Gamma_{i} \in \Rm{L}_{\Rm{ta}}(R_{\dagger},\Gamma_{i},z^{\dagger}) \subset 
\Rm{L}_{\Rm{ta}}(R_{\dagger},z^{\dagger})$.}

To show Proposition~\ref{estimateK1},
we consider \Blue{the following four conditions for} $\Gamma_{p^{\dagger}}$ where $R_{\dagger}$ has a singular point.
\begin{itemize}
\item[(C1)] There exists a singular point $z^{\dagger} \in R_{\dagger}$ on $\Gamma_{p^{\dagger}}$ such that $\Rm{Stab}_{T(\Ti{K}^{\dagger})_{p^{\dagger}}}(z^{\dagger})=\{\Rm{id}\}$.

\item[(C2)] The condition $(C1)$ does not hold and there exist a singular point $z^{\dagger}$ of $(R_{\dagger})$ on $\Gamma_{p^{\dagger}}$ and a local analytic branch $\Gamma_{i}$ of $\Gamma_{p^{\dagger}}$ around
$z^{\dagger}$ such that $\Rm{L}_{\Rm{ta}}((R_{\dagger})_{h},\Gamma_{i},z^{\dagger}) \neq \emptyset$.

\item[(C3)] Conditions $(C1)$ and $(C2)$ do not hold and there \Red{exists} a singular point $z^{\dagger} \in R_{\dagger}$ on $\Gamma_{p}$ such that $ \in \Rm{L}_{\Rm{tr}}((R_{\dagger})_{h},z^{\dagger}) \neq \emptyset$.

\item[(C4)] Conditions $(C1)$, $(C2)$ and $(C3)$ do not hold.
\end{itemize}

\begin{prop}
Assume $r \geq 4n$ and $(C1)$ holds.
Then it holds that
\begin{align*}
2 \delta n K_{f}^{2}(\Gamma_{p}) \geq &2n \left(n^{2}-1-n \right) \sharp \Ti{K}.
\end{align*}
\end{prop}

\begin{proof}
Let $z^{\dagger}$ be \Red{a singular point} of $R_{\dagger}$ on $\Gamma_{p^{\dagger}}$.
By Lemma~\ref{lifting2}, there exist $\sharp T(\Ti{K}^{\dagger})_{p^{\dagger}}l$ singular points of \Red{the branch loci} on $\Gamma_{p^{\dagger}}$ whose multiplicity is at least $n$.

If there \Blue{exists} a singular point of \Red{the branch loci} over $z^{\dagger} \in R_{\dagger}$,
there exist $2\sharp(T(\Ti{K}^{\dagger})_{p^{\dagger}})l$ singular points of $R_{\dagger}$
on $\Gamma_{p}$ whose multiplicity is at least $n$.
\Blue{Applying} Proposition~\ref{localization'} for $\Ch{\psi}_{\Delta^{\dagger}}=\Rm{id}$ and 
Lemma~\ref{lowerbdK1'}, 
we have
\begin{align*}
 K_{f}^2 (\Gamma_{p^{\dagger}})& \geq 
  2 \left((n^{2}-1)-n-\frac{(n-1)^{2}}{n}B_{n} \right) \sharp T(\Ti{K}^{\dagger})_{p^{\dagger}} l  \\
 &-2\frac{(n-1)^{2}}{n}j_{0,\bullet}'(\Gamma_{p^{\dagger}})
 +\frac{n^{2}-1}{n}r\chi_{\varphi_{\Delta^{\dagger}}}(\Gamma_{p^{\dagger}}) +j_{0,1}'(\Gamma_{p^{\dagger}})\\
 &\geq  2 \left((n^{2}-1)-n-\frac{(n-1)^{2}}{n}B_{n} \right) \sharp T(\Ti{K}^{\dagger})_{p^{\dagger}} l.
 \end{align*}
Hence it is sufficient to show 
\begin{align*}
4\delta n \left((n^{2}-1)-n-\frac{(n-1)^{2}}{n}B_{n} \right) \sharp T(\Ti{K}^{\dagger})_{p^{\dagger}}
-2n\left(n^{2}-1-n \right) \sharp \Ti{K} \geq 0.
\end{align*}
From Corollary~\ref{groupstr}, we have 
$\sharp \Ti{K} \leq \delta \sharp  T(\Ti{K}^{\dagger})_{p^{\dagger}}$.
Hence it is sufficient to show  
\begin{align*}
2 \left((n^{2}-1)-n-\frac{(n-1)^{2}}{n}B_{n} \right) -\left(n^{2}-1-n \right) \geq 0.
\end{align*}
We can check it by a simple calculation.

Suppose that there exist no singular points over $z^{\dagger} \in R_{\dagger}$.
We recall that there exist $\sharp(T(\Ti{K}^{\dagger})_{p^{\dagger}})l$ singular points of 
$R_{\dagger}$ on $\Gamma_{p}$ which is \Blue{analytically equivalent} to $z^{\dagger} \in R_{\dagger}$.
Let $\Ch{\psi}_{\Delta^{\dagger}}:\Ha{W}_{\Delta^{\dagger}}\to W_{\Delta^{\dagger}}$ 
be the composite of blowing-ups at those $\sharp(T(\Ti{K}^{\dagger})_{p^{\dagger}})l$ singular points.
Since there exist no singular points over $z^{\dagger} \in R_{\dagger}$,
we have $\Ch{j}_{0,\bullet}(\Gamma_{p^{\dagger}})=0$.
From Proposition~\ref{localization'} for $\Ch{\psi}_{\Delta^{\dagger}}$,
we get
\begin{align*}
 K_{f}^2 (\Gamma_{p^{\dagger}}) = &\sum_{\Ha{z}^{\dagger}\in \Ha{\Gamma}_{p^{\dagger}}} K_{f}^2 (\Gamma_{p^{\dagger}})_{\Ha{z}^{\dagger}} 
 + \left(n^{2}-1-n \right)\sharp(T(\Ti{K}^{\dagger})_{p^{\dagger}})l \\
 &-2\frac{(n-1)^{2}}{n}j_{0,\bullet}'(\Gamma_{p^{\dagger}})
 +\frac{n^{2}-1}{n}r\chi_{\varphi_{\Delta^{\dagger}}}(\Gamma_{p^{\dagger}}) +j_{0,1}'(\Gamma_{p^{\dagger}}).
 \end{align*}
Since we have $\sharp \Gamma_{p^{\dagger}} \geq j_{0,1}'(\Gamma_{p^{\dagger}})$ 
and $\chi_{\varphi_{\Delta^{\dagger}}}(\Gamma_{p^{\dagger}}) \geq \sharp \Gamma_{p^{\dagger}}/12$, we have
\begin{align*}
 K_{f}^2 (\Gamma_{p^{\dagger}}) &\geq 
  \left(n^{2}-1-n \right)\sharp(T(\Ti{K}^{\dagger})_{p^{\dagger}})l 
 +\frac{n^{2}-1}{n}\left(r-24\frac{n-1}{n+1}\right)\sharp\Gamma_{p^{\dagger}}\\
 &\geq \left(n^{2}-1-n \right)\sharp(T(\Ti{K}^{\dagger})_{p^{\dagger}})l .
 \end{align*}
Hence we have 
\begin{align*}
 K_{f}^2 (\Gamma_{p}) \geq \left(n^{2}-1-n \right)\sharp(T(\Ti{K}^{\dagger})_{p^{\dagger}}) .
\end{align*}
Thus, it is sufficient to show 
\begin{align*}
2\delta n  \left(n^{2}-1-n \right) \sharp T(\Ti{K}^{\dagger})_{p^{\dagger}}
-2n\left(n^{2}-1-n \right) \sharp \Ti{K} \geq 0.
\end{align*}
From Corollary~\ref{groupstr}, we have 
$\sharp \Ti{K} \leq \delta \sharp  T(\Ti{K}^{\dagger})_{p^{\dagger}}$.
Hence we get the desired inequality.
\end{proof}

\begin{rmk}
If (C1) does not hold, there exist no singular points of $R_{\dagger}$
on $\Gamma_{p,\Rm{sm}}$ by Corollary~\ref{stabfree}.
\end{rmk}

\begin{prop}
Assume $r \geq \Rm{max}\{60 + \frac{12}{n^2-1}-\frac{96}{n+1}, 4n\}$ and $(C2)$ holds.
Put 
\Blue{\[C_{2,n}:=\frac{1}{3}(n-1)(5n-4).\]}
Then it holds that
\begin{align*}
 2n \delta  K_{f}^{2}(\Gamma_{p})  \geq C_{2,n} \sharp \Ti{K}.
\end{align*}
\end{prop}

\begin{proof}

Take $z^{\dagger} \in R_{\dagger}$ and $\Gamma_{1} \subset \Gamma_{p^{\dagger}}$ such that  
$L_{\Rm{ta}}((R_{\dagger})_{h},\Gamma_{1},z^{\dagger}) \neq \emptyset$.
\Blue{
For a local analytic branch $D \in L_{\Rm{ta}}((R_{\dagger})_{h},\Gamma_{1},z^{\dagger})$,
let $v(D)$ be the number of blowing-ups until the proper transform of $D$ is not tangent to $\Gamma_1$.
Since $L_{\Rm{ta}}((R_{\dagger})_{h},\Gamma_{1},z^{\dagger}) \neq \emptyset$ is a finite set,
there exists $D_{1} \in L_{\Rm{ta}}((R_{\dagger})_{h},\Gamma_{1},z^{\dagger})$ such that 
$v(D_1)$ is a minimal value among $L_{\Rm{ta}}((R_{\dagger})_{h},\Gamma_{1},z^{\dagger})$.}
Let $\Rm{Stab}_{T(\Ti{K}^{\dagger})_{p^{\dagger}}}(\Gamma_{1},z^{\dagger})
:=\{\kappa \in \Rm{Stab}_{T(\Ti{K}^{\dagger})_{p^{\dagger}}}(z^{\dagger})\mid \kappa(\Gamma_{1})=\Gamma_{1}\}$ and 
$\Bb{D}_{1}$ be the  $\Rm{Stab}_{T(\Ti{K}^{\dagger})_{p^{\dagger}}}(\Gamma_{1},z^{\dagger})$-orbit of $D_{1}$.
We introduce some notations as follows.
\begin{itemize}
\item $m_{D_{1}}:=\Rm{mult}_{z^{\dagger}}D_{1}$
\item $m_{\Bb{D}_{1}}:=\Rm{mult}_{z^{\dagger}}\Bb{D}_{1}$
\item $m:=\Rm{min}\{m' \mid m' \geq m_{\Bb{D}},\quad m' \in n\Z$ or $n\Z + 1\}$
\end{itemize}
Assume that we need \Red{blowing-ups $v$ times} until the proper transform of $D_{1}$ is not tangent to that of $\Gamma_{1}$.
Since $\Gamma_{p^{\dagger}}$ has a node at $z^{\dagger}$, we have $(\Bb{D}_{1},\Gamma_{p^{\dagger}})_{z^{\dagger}} \leq (v+2)m$.

From Lemma~\ref{lifting2}, there exist
$\sharp ({T(\Ti{K}^{\dagger})_{p^{\dagger}}} \cdot z^{\dagger}) l$
local analytic curves of $R_{\dagger}$ 
which is \Blue{analytically equivalent} to $z^{\dagger} \in \Bb{D}_{1}$.
Let $z^{\dagger}_{j} \in \Bb{D}_{j}$ be such local analytic curves for 
$j=1, \cdots, ({T(\Ti{K}^{\dagger})_{p^{\dagger}}} \cdot z^{\dagger}) l$,
where we define $z^{\dagger}_{1} = z^{\dagger}$.
Since $z^{\dagger}_{j} \in \Bb{D}_{j}$ is analytically equivalent to $z^{\dagger}_{1} \in \Bb{D}_{1}$,
the proper transform of $\Bb{D}_{j}$ is not tangent to 
the proper transform of $\Gamma_{p^{\dagger}}$ after $v$-times blowing-ups over $z^{\dagger}_{j}$ for each
$j=1, \cdots, ({T(\Ti{K}^{\dagger})_{p^{\dagger}}} \cdot z^{\dagger}) l$.

We consider two cases separately.

\smallskip

(1)~$n \geq 3$ or $n=2$ and $\Rm{mult}_{z^{\dagger}} R_{\dagger} \in 2\Z$.

Let $\Ch{\psi}_{\Delta^{\dagger}}:\Ha{W}_{\Delta^{\dagger}}\to W_{\Delta^{\dagger}}$ 
be the composite of the above $v$-times blowing-ups over $z^{\dagger}_{j}$ for 
each $j=1, \cdots, ({T(\Ti{K}^{\dagger})_{p^{\dagger}}} \cdot z^{\dagger}) l$.
Let $\Ha{R}_{\dagger}$ be the branch locus on $\Ha{W}_{\Delta^{\dagger}}$ and
 \Red{$\Ha{\Gamma}_{p^{\dagger}}$ a fiber} of 
 $\Ha{\varphi}:=\varphi_{\Delta^{\dagger}} \circ \Ch{\psi}_{\Delta^{\dagger}}$ over $p^{\dagger}$.
 We show 
\begin{align*}
 \Ch{j}_{0,\bullet}(\Gamma_{p^{\dagger}}) \leq
\sharp(T(\Ti{K}^{\dagger})_{p^{\dagger}} \cdot z^{\dagger}) l
\end{align*}
if $n\geq 3$,
\begin{align*}
 \Ch{j}_{0,\bullet}(\Gamma_{p^{\dagger}}) \leq
 \Ch{j}_{0,1}(\Gamma_{p^{\dagger}}) 
+\sharp(T(\Ti{K}^{\dagger})_{p^{\dagger}} \cdot z^{\dagger}) l 
\end{align*}
if $n=2$.
We may assume $v \geq 2$.
Let $\Ch{j}_{0,\bullet}(\Gamma_{p^{\dagger}})_{z^{\dagger}}$
be the number of irreducible components contributing to
$\Ch{j}_{0,\bullet}(\Gamma_{p^{\dagger}})$ which is contracted to $z^{\dagger}$ by $\Ti{\psi}_{\Delta^{\dagger}}$.
If $n\geq 3$, it is sufficient to show 
$\Ch{j}_{0,\bullet}(\Gamma_{p^{\dagger}})_{z^{\dagger}} \leq 2$.
\Blue{
Let $(\Ch{\psi}_{\Delta^{\dagger}})_{z^{\dagger}}:(\Ha{W}_{\Delta^{\dagger}})_{z^{\dagger}}\to W_{\Delta^{\dagger}}$ be the composite of blowing-ups $v$-times over $z^{\dagger}$ 
and $\Ca{E}_{z^{\dagger}}$ the set of exceptional curves of $(\Ch{\psi}_{\Delta^{\dagger}})_{z^{\dagger}}$.
We note that $\sharp\Ca{E}_{z^{\dagger}}=v$.
For an exceptional curve $E \in \Ca{E}_{z^{\dagger}}$, let $n(E)$ be the number such that 
$E$ is the exceptional curve of the $n(E)$-th blowing-up in $(\Ch{\psi}_{\Delta^{\dagger}})_{z^{\dagger}}$.
Let $E_i$ be the exceptional curve in $\Ca{E}_{z^{\dagger}}$ 
such that $n(E_i)=i$ for $i =1, \cdots v$.
We note that $E_1$ and $E_v$ are exceptional curves of first and last blowing-ups in $(\Ch{\psi}_{\Delta^{\dagger}})_{z^{\dagger}}$, respectively.
Then the proper transform of any local analytic branch in $\Rm{L}((R_{\dagger})_{h},z^{\dagger})\setminus\Rm{L}_{\Rm{ta}}(R_{\dagger},\Gamma_{i},z^{\dagger})$ can be tangent to only $E_1$ among $\Ca{E}_{z^{\dagger}}$.
Furthermore, the proper transform of any local analytic branch in $\Rm{L}_{\Rm{ta}}(R_{\dagger},\Gamma_{i},z^{\dagger})$ can be tangent to only $E_{v}$ from
the choice of $D_1$.
Thus, the self-intersection number of $\Ha{E}_{i}(i=2, \cdots, v-1)$ is $-2$,
where $\Ha{E}_{i} \subset (\Ha{W}_{\Delta})_{z^{\dagger}}$ is the proper transform of $E_{i}$.
On the other hand, the self-intersection number of each irreducible component of $\Ti{R}_{\dagger}$
contributing $\Ch{j}_{0,\bullet}(\Gamma_{p^{\dagger}})_{z^{\dagger}}$ is at most $-n$.
Thus, if $n \geq 3$, a proper transform of $\Ca{E}_{z^{\dagger}}$ which is contained in $\Ti{R}_{\dagger}$ is either $\Ha{E}_1$ or $\Ha{E}_v$ at most.
Hence we get $\Ch{j}_{0,\bullet}(\Gamma_{p^{\dagger}})_{z^{\dagger}} \leq 2$ for $n \geq 3$.
}
Furthermore, since \Red{the branch loci} have no double points, we have  
$\Ch{j}_{0,\bullet}(\Gamma_{p^{\dagger}})_{z^{\dagger}} \leq 1$
by chasing a resolution of $z^{\dagger} \in R_{\dagger}$.
If $n=2$, we can show 
$\Ch{j}_{0,\bullet}(\Gamma_{p^{\dagger}})_{z^{\dagger}} \leq \Ch{j}_{0,1}(\Gamma_{p^{\dagger}})_{z^{\dagger}} +1$
by chasing a resolution of $z^{\dagger} \in R_{\dagger}$.
Thus, we can get \Blue{the desired} inequalities for $\Ch{j}_{0,\bullet}(\Gamma_{p^{\dagger}})$.

Since $\Gamma_{p^{\dagger}}$ has at least 
$\sharp({T(\Ti{K}^{\dagger})_{p^{\dagger}}} \cdot z^{\dagger}) l$ nodes,
we have $\sharp\Gamma_{p^{\dagger}} \geq \sharp({T(\Ti{K}^{\dagger})_{p^{\dagger}}} \cdot z^{\dagger}) l$.
Hence we get 
\begin{align*}
\Ch{j}_{0,\bullet}(\Gamma_{p^{\dagger}}) \leq
\sharp\Gamma_{p^{\dagger}}
\end{align*}
if $n\geq 3$,
\begin{align*}
 \Ch{j}_{0,\bullet}(\Gamma_{p^{\dagger}}) \leq
 \sharp\Gamma_{p^{\dagger}}+\Ch{j}_{0,1}(\Gamma_{p^{\dagger}}) 
\end{align*}
if $n=2$.
Put $\Bb{D}:=\Rm{Stab}_{T(\Ti{K}^{\dagger})_{p^{\dagger}}}(z^{\dagger}) \cdot D_{1}$
and $\gamma:=1-1/(D_{1},\Gamma_{p^{\dagger}})_{z^{\dagger}}$.
Then we have 
\begin{align*}
\Blue{\sum_{\Ha{z}^{\dagger}\in \Ha{\Gamma}_{p^{\dagger}}}\alpha^{+}_{0}(\Gamma_{p^{\dagger}})_{\Ha{z}^{\dagger}}
\geq \gamma (\Bb{D},\Gamma_{p^{\dagger}})_{z^{\dagger}} 
\sharp({T(\Ti{K}^{\dagger})_{p^{\dagger}}} \cdot z^{\dagger}) l.}
\end{align*}
Since there exist $\sharp ({T(\Ti{K}^{\dagger})_{p^{\dagger}}} \cdot z^{\dagger}) vl$-singular points of \Red
{the branch loci} on $\Gamma_{p^{\dagger}}$ which appear in $\Ch{\psi}_{\Delta^{\dagger}}$, we have 
\begin{align*}
\sum_{k\geq 1}\left((n^{2}-1)k-n \right)\Ch{\alpha}_{k} (\Gamma_{p^{\dagger}})
\geq \left((n^{2}-1)\left[ \frac{m}{n}\right]-n \right) \sharp ({T(\Ti{K}^{\dagger})_{p^{\dagger}}} \cdot z^{\dagger})vl.
\end{align*}
Thus, we have 
\begin{align*}
K_{f}^2 (\Gamma_{p^{\dagger}}) \geq &
 \gamma \frac{(n-1)^2}{n} (\Bb{D},\Gamma_{p^{\dagger}})_{z^{\dagger}} 
\sharp({T(\Ti{K}^{\dagger})_{p^{\dagger}}} \cdot z^{\dagger}) l \\
+&\left((n^{2}-1)\left[ \frac{m}{n}\right]-n \right) \sharp ({T(\Ti{K}^{\dagger})_{p^{\dagger}}} \cdot z^{\dagger})vl \\
-&2\frac{(n-1)^2}{n}\left(\Ch{j}_{0,\bullet}(\Gamma_{p^{\dagger}})+j_{0,\bullet}'(\Gamma_{p^{\dagger}})\right)
+\frac{n^{2}-1}{n}r\chi_{\varphi_{\Delta^{\dagger}}}(\Gamma_{p^{\dagger}})  + \Ch{j}_{0,1}(\Gamma_{p^{\dagger}}) + j_{0,1}'(\Gamma_{p^{\dagger}})
\end{align*}
by Lemma~\ref{6/15,1}.
By $\chi_{\varphi_{\Delta^{\dagger}}}(\Gamma_{p^{\dagger}})=e(\Gamma_{p^{\dagger}})/12 = 
\sharp \Gamma_{p^{\dagger}} /12
\geq \sharp (T(\Ti{K}^{\dagger})_{p^{\dagger}} \cdot z^{\dagger}) l /12$,
we have 
\begin{align*}
K_{f}^2 (\Gamma_{p^{\dagger}}) \geq &
 \gamma \frac{(n-1)^2}{n} (\Bb{D},\Gamma_{p^{\dagger}})_{z^{\dagger}} 
\sharp({T(\Ti{K}^{\dagger})_{p^{\dagger}}} \cdot z^{\dagger}) l \\
+&\left((n^{2}-1)\left[ \frac{m}{n}\right]-n \right) \sharp ({T(\Ti{K}^{\dagger})_{p^{\dagger}}} \cdot z^{\dagger})vl \\
+&\left( -4\frac{(n-1)^2}{n}+\frac{n^2-1}{12n}r\right) 
\sharp ({T(\Ti{K}^{\dagger})_{p^{\dagger}}} \cdot z^{\dagger})l.
\end{align*}
By $K_{f}^{2}(\Gamma_{p}) \geq K_{f}^{2}(\Gamma_{p^{\dagger}})/l$, we get
\begin{align*}
K_{f}^2 (\Gamma_{p}) \geq &
 \gamma \frac{(n-1)^2}{n} (\Bb{D},\Gamma_{p^{\dagger}})_{z^{\dagger}} 
\sharp({T(\Ti{K}^{\dagger})_{p^{\dagger}}} \cdot z^{\dagger})  \\
+&\left((n^{2}-1)\left[ \frac{m}{n}\right]-n \right) \sharp ({T(\Ti{K}^{\dagger})_{p^{\dagger}}} \cdot z^{\dagger})v \\
+&\left( -4\frac{(n-1)^2}{n}+\frac{n^2-1}{12n}r\right) 
\sharp ({T(\Ti{K}^{\dagger})_{p^{\dagger}}} \cdot z^{\dagger}).
\end{align*}
Hence it is sufficient to show
\begin{align*}
2n\delta \biggl\{ &\gamma \frac{(n-1)^2}{n} (\Bb{D},\Gamma_{p^{\dagger}})_{z^{\dagger}}\biggr.
+\left((n^{2}-1)\left[ \frac{m}{n}\right]-n \right)v\\
+\biggl. &\left( -4\frac{(n-1)^2}{n}+\frac{n^2-1}{12n}r\right) \biggr\}
\sharp ({T(\Ti{K}^{\dagger})_{p^{\dagger}}} \cdot z^{\dagger}) 
-C_{2,n} \sharp \Ti{K} \geq 0.
\end{align*}
By $(\Bb{D},\Gamma_{p^{\dagger}})_{z^{\dagger}} \geq \sharp \Rm{Stab}_{T(\Ti{K}^{\dagger})_{p^{\dagger}}}(z^{\dagger}) $, it is sufficient to show
\begin{align*}
2n \biggl\{ &\gamma \frac{(n-1)^2}{n} (\Bb{D},\Gamma_{p^{\dagger}})_{z^{\dagger}}\biggr.
+\left((n^{2}-1)\left[ \frac{m}{n}\right]-n \right)v\\
+\biggl. &\left( -4\frac{(n-1)^2}{n}+\frac{n^2-1}{12n}r\right) \biggr\}
-C_{2,n}(\Bb{D},\Gamma_{p^{\dagger}})_{z^{\dagger}}   \geq 0.
\end{align*}
It is equivalent to 
\begin{align*}
2n \left( \left((n^{2}-1)\left[ \frac{m}{n}\right]-n \right)v
-4\frac{(n-1)^2}{n}+\frac{n^2-1}{12n}r \right) 
+\left(2 \gamma(n-1)^2  -C_{2,n}\right)(\Bb{D},\Gamma_{p^{\dagger}})_{z^{\dagger}} \geq 0.
\end{align*}
By $2m(v+2)\geq 2(\Bb{D}_{1},\Gamma_{p^{\dagger}})_{z^{\dagger}} 
\geq (\Bb{D},\Gamma_{p^{\dagger}})_{z^{\dagger}}$,
it suffices to show
\begin{align}
\label{6/9,1}
2n \left( \left((n^{2}-1)\left[ \frac{m}{n}\right]-n \right)v
-4\frac{(n-1)^2}{n}+\frac{n^2-1}{12n}r \right) 
+2m(v+2)\left(2 \gamma(n-1)^2  -C_{2,n}\right) \geq 0.
\end{align}
Since it holds $(D_{1},\Gamma_{p^{\dagger}})_{z^{\dagger}}\geq 3$,
we have \Blue{$\gamma = 1-\frac{1}{(D_{1},\Gamma_{p^{\dagger}})_{z^{\dagger}}} \geq 2/3$}.
Hence \Blue{it suffices to show}
\begin{align}
\label{6/8,1}
2n \left( \left((n^{2}-1)\left[ \frac{m}{n}\right]-n \right)v
-4\frac{(n-1)^2}{n}+\frac{n^2-1}{12n}r \right)+m(v+2)\left(\frac{8}{3}(n-1)^2  -2C_{2,n}\right)\geq 0.
\end{align}
The coefficient of $v$ in (\ref{6/8,1}) 
\begin{align*}
2n\left((n^{2}-1)\left[ \frac{m}{n}\right]-n \right)+\left(\frac{8}{3}(n-1)^2  -2C_{2,n}\right)m
\end{align*}
is non-negative.
Hence we may assume $v=1$ to show (\ref{6/8,1}).
We will show 
\begin{align}
\label{6/8,2}
2n \left( \left((n^{2}-1)\left[ \frac{m}{n}\right]-n \right)
-4\frac{(n-1)^2}{n}+\frac{n^2-1}{12n}r \right)+m\left(8(n-1)^2  -6C_{2,n}\right)\geq 0.
\end{align}
Since the coefficient of $m$ is non-negative,
we may assume $m=n+1$ to show (\ref{6/8,2}).
Hence it suffices to show 
\begin{align*}
 &2n(n^2-1)+(n+1)\left(8(n-1)^2  -6C_{2,n}\right)\geq 0, \\
 &\frac{n^2-1}{12n}r -4\frac{(n-1)^2}{n}-n \geq 0.
\end{align*}
We can check the former inequality by \Red{a simple calculation}.
The latter one clearly holds by the assumption $r \geq 60 + \frac{12}{n^2-1}-\frac{96}{n+1}$.

\smallskip

(2)~$n=2$ and $\Rm{mult}_{z^{\dagger}}R_{\dagger} \in 2\Z+1$.

Let $(\Ch{\psi}_{\Delta^{\dagger}})_{0}:(\Ha{W}_{\Delta^{\dagger}})_{0}\to W_{\Delta^{\dagger}}$ be the composite of the above \Blue{blowing-ups $v$-times} over $z^{\dagger}_{j}$ for each
$j=1, \cdots, ({T(\Ti{K}^{\dagger})_{p^{\dagger}}} \cdot z^{\dagger}) l$.
Let $(\Ha{R}_{\dagger})_{0}$ be the branch locus on $(\Ha{W}_{\Delta^{\dagger}})_{0}$.
Then all exceptional curves appearing from $(\Ch{\psi}_{\Delta^{\dagger}})_{0}$
are contained in $(\Ha{R}_{\dagger})_{0}$ by the assumption (2).
Thus, there exist $\sharp ({T(\Ti{K}^{\dagger})_{p^{\dagger}}} \cdot z^{\dagger})(v-1)l$
double points of $(\Ha{R}_{\dagger})_{0}$. 
Let $\Ch{\psi}_{\Delta^{\dagger}}:\Ha{W}_{\Delta^{\dagger}}\to W_{\Delta^{\dagger}}$ be the composite of blowing-ups at those double points and $(\Ch{\psi}_{\Delta^{\dagger}})_{0}$.
Let $\Ha{R}_{\dagger}$ be the branch locus on $\Ha{W}_{\Delta^{\dagger}}$ and
\Red{ $\Ha{\Gamma}_{p^{\dagger}}$ a fiber} of 
 $\Ha{\varphi}:=\varphi_{\Delta^{\dagger}} \circ \Ch{\psi}_{\Delta^{\dagger}}$ 
 over $p^{\dagger}$.
 Since an arbitrary $\Ha{\varphi}_{\Delta^{\dagger}}$-vertical component of $\Ha{R}_{\dagger}$
 appears from either  exceptional curves in $(\Ch{\psi}_{\Delta^{\dagger}})_{0}$ or  irreducible components of $\Gamma_{p^{\dagger}}$,
 we have
\begin{align*}
 \Ch{j}_{0,\bullet}(\Gamma_{p^{\dagger}}) \leq
 \sharp(T(\Ti{K}^{\dagger})_{p^{\dagger}} \cdot z^{\dagger})v l.
\end{align*}
Put $\Bb{D}:=\Rm{Stab}_{T(\Ti{K}^{\dagger})_{p^{\dagger}}}(z^{\dagger}) \cdot D_{1}$
and $\gamma:=1-1/(D_{1},\Gamma_{p^{\dagger}})_{z^{\dagger}}$.
Then we have 
\begin{align*}
\sum_{\Ha{z}\in \Ha{\Gamma}_{p}}\alpha^{+}_{0}(\Gamma_{p^{\dagger}})_{\Ha{z}}
\geq \gamma (\Bb{D},\Gamma_{p^{\dagger}})_{z^{\dagger}} 
\sharp({T(\Ti{K}^{\dagger})_{p^{\dagger}}} \cdot z^{\dagger}) l.
\end{align*}
Since $R_{\dagger}$ has 
$\sharp ({T(\Ti{K}^{\dagger})_{p^{\dagger}}} \cdot z^{\dagger}) vl$-singular points 
of multiplicity at least $m$ and 
$\sharp ({T(\Ti{K}^{\dagger})_{p^{\dagger}}} \cdot z^{\dagger})(v-1)l$
double points which appear in $(\Ch{\psi}_{\Delta^{\dagger}})_{0}$,
we have 
\begin{align*}
\sum_{k\geq 1}\left(3k-2 \right)\Ch{\alpha}_{k} (\Gamma_{p^{\dagger}})
\geq \left(3\left[ \frac{m}{2}\right]-2 \right) \sharp ({T(\Ti{K}^{\dagger})_{p^{\dagger}}} \cdot z^{\dagger})vl
+\sharp ({T(\Ti{K}^{\dagger})_{p^{\dagger}}} \cdot z^{\dagger})(v-1)l.
\end{align*}
 Thus, we have
 \begin{align*}
K_{f}^2 (\Gamma_{p^{\dagger}}) \geq &
 \frac{\gamma}{2} (\Bb{D},\Gamma_{p^{\dagger}})_{z^{\dagger}} 
\sharp({T(\Ti{K}^{\dagger})_{p^{\dagger}}} \cdot z^{\dagger}) l \\
+&\left(3\left[ \frac{m}{2}\right]-2 \right) \sharp ({T(\Ti{K}^{\dagger})_{p^{\dagger}}} \cdot z^{\dagger})vl
+\sharp ({T(\Ti{K}^{\dagger})_{p^{\dagger}}} \cdot z^{\dagger})(v-1)l\\
-&\left(\Ch{j}_{0,\bullet}(\Gamma_{p^{\dagger}})+j_{0,\bullet}'(\Gamma_{p^{\dagger}})\right)+
\frac{3}{2}r\chi_{\varphi_{\Delta^{\dagger}}}(\Gamma_{p^{\dagger}})  + \Ch{j}_{0,1}(\Gamma_{p^{\dagger}}) + \j_{0,1}'(\Gamma_{p^{\dagger}})
\end{align*}
by Lemma~\ref{6/15,1}.
By $\chi_{\varphi_{\Delta^{\dagger}}}(\Gamma_{p^{\dagger}})\geq \sharp (T(\Ti{K}^{\dagger})_{p^{\dagger}} \cdot z^{\dagger}) l /12$
and $2\sharp \Gamma_{p^{\dagger}}+(v-1)l\sharp (T(\Ti{K}^{\dagger})_{p^{\dagger}} \cdot z^{\dagger})  \geq \left(\Ch{j}_{0,\bullet}(\Gamma_{p^{\dagger}})+j_{0,\bullet}'(\Gamma_{p^{\dagger}})\right)$,
we have 
 \begin{align*}
K_{f}^2 (\Gamma_{p^{\dagger}}) \geq &
  \frac{\gamma}{2} (\Bb{D},\Gamma_{p^{\dagger}})_{z^{\dagger}} 
\sharp({T(\Ti{K}^{\dagger})_{p^{\dagger}}} \cdot z^{\dagger}) l \\
+&\left(3\left[ \frac{m}{2}\right]-n \right) \sharp ({T(\Ti{K}^{\dagger})_{p^{\dagger}}} \cdot z^{\dagger})vl
+\left(\frac{1}{8}r-2 \right)\sharp ({T(\Ti{K}^{\dagger})_{p^{\dagger}}} \cdot z^{\dagger})l.
\end{align*}
Hence it is sufficient to show 
\begin{align*}
4 \left( \left(3\left[ \frac{m}{2}\right]-2 \right)v
+\frac{1}{8}r-2 \right) 
+2m(v+2)\left(2\gamma  -C_{2,2}\right) \geq 0.
\end{align*}
\Red{It clearly holds} by (\ref{6/9,1}).
\end{proof}

\begin{prop}
Assume $r \geq \Rm{max}\{60 + \frac{12}{n^2-1}-\frac{96}{n+1}, 4n\}$ and $(C3)$ holds.
Put
\[C_{3,n}:=(n-1)(2n-1).\]
Then it holds that
\begin{align*}
2\delta nK_{f}^{2}(\Gamma_{p}) \geq C_{3,n}\sharp \Ti{K}.
 \end{align*}
\end{prop}

\begin{proof}
Take $z^{\dagger} \in R_{\dagger}$ and 
$D \in L_{\Rm{tr}}(R_{\dagger},z^{\dagger})$.
Denote by $\Bb{D}$ the 
$\Rm{Stab}_{T(\Ti{K}^{\dagger})_{p^{\dagger}}}(z^{\dagger}$)-orbit of $D$.
We introduce the following notations.
\begin{itemize}
\item $m_{\Bb{D}}:=\Rm{mult}_{z^{\dagger}}\Bb{D}$.
\item $m:=\Rm{min}\{m' \mid m' \geq m_{\Bb{D}}, m' \in n\Z_{>0}$ or $n\Z_{>0}+1\}$.
\end{itemize}
From Lemma~\ref{lifting2},
there exist $\sharp (T(\Ti{K}^{\dagger})_{p^{\dagger}}\cdot z^{\dagger})l$ singular points
of \Red{the branch loci} on $\Gamma_{p^{\dagger}}$ which is \Blue{analytically equivalent} to 
$z^{\dagger} \in \Bb{D}$.
Let $\Ch{\psi}_{\Delta^{\dagger}}:\Ha{W}_{\Delta^{\dagger}}\to W_{\Delta^{\dagger}}$ be the composite of blowing-ups at $z^{\dagger}_{j}$ ($j=1, \cdots, ({T(\Ti{K}^{\dagger})_{p^{\dagger}}} \cdot z^{\dagger}) l$).
Let $\Ha{R}_{\dagger}$ be the branch locus on $\Ha{W}_{\Delta^{\dagger}}$ and
\Red{ $\Ha{\Gamma}_{p^{\dagger}}$ a fiber} of 
 $\Ha{\varphi}:=\varphi_{\Delta^{\dagger}} \circ \Ch{\psi}_{\Delta^{\dagger}}$ 
 over $p^{\dagger}$.
Put $\gamma:=1-1/(D_{1},\Gamma_{p^{\dagger}})_{z^{\dagger}}$.
Then we have 
\begin{align*}
K_{f}^2 (\Gamma_{p^{\dagger}}) \geq &
 \gamma \frac{(n-1)^2}{n} (\Bb{D},\Gamma_{p^{\dagger}})_{z^{\dagger}} 
\sharp({T(\Ti{K}^{\dagger})_{p^{\dagger}}} \cdot z^{\dagger}) l \\
+&\left((n^{2}-1)\left[ \frac{m}{n}\right]-n \right) \sharp ({T(\Ti{K}^{\dagger})_{p^{\dagger}}} \cdot z^{\dagger})l \\
-&2\frac{(n-1)^2}{n}\left(\Ch{j}_{0,\bullet}(\Gamma_{p^{\dagger}})+j_{0,\bullet}'(\Gamma_{p^{\dagger}})\right)+
\frac{n^{2}-1}{n}r\chi_{\varphi_{\Delta^{\dagger}}}(\Gamma_{p^{\dagger}}) .
\end{align*}
By $\chi_{\varphi_{\Delta^{\dagger}}}(\Gamma_{p^{\dagger}})\geq \sharp \Gamma_{p^{\dagger}} /12$,
$\sharp \Gamma_{p^{\dagger}} \geq \sharp (T(\Ti{K}^{\dagger})_{p^{\dagger}} \cdot z^{\dagger}) l$
and $2\sharp \Gamma_{p^{\dagger}} \geq \left(\Ch{j}_{0,\bullet}(\Gamma_{p^{\dagger}})+j_{0,\bullet}'(\Gamma_{p^{\dagger}})\right)$,
we have 
\begin{align*}
K_{f}^2 (\Gamma_{p}) \geq &
 \gamma \frac{(n-1)^2}{n} (\Bb{D},\Gamma_{p^{\dagger}})_{z^{\dagger}} 
\sharp({T(\Ti{K}^{\dagger})_{p^{\dagger}}} \cdot z^{\dagger})  \\
+&\left((n^{2}-1)\left[ \frac{m}{n}\right]-n \right) \sharp ({T(\Ti{K}^{\dagger})_{p^{\dagger}}} \cdot z^{\dagger}) \\
+&\left( -4\frac{(n-1)^2}{n}+\frac{n^2-1}{12n}r\right) 
\sharp ({T(\Ti{K}^{\dagger})_{p^{\dagger}}} \cdot z^{\dagger}).
\end{align*}
Thus, it is sufficient to show
\begin{align*}
2n \left( \left((n^{2}-1)\left[ \frac{m}{n}\right]-n \right)
-4\frac{(n-1)^2}{n}+\frac{n^2-1}{12n}r \right) 
+\left(2 \gamma(n-1)^2  -C_{3,n}\right)(\Bb{D},\Gamma_{p^{\dagger}})_{z^{\dagger}} \geq 0.
\end{align*}
\Blue{Since $D$ is in $L_{\Rm{tr}}((R_{\dagger})_{h},z^{\dagger})$ and 
$\Gamma_{p^{\dagger}}$ has nodes at $z^{\dagger}$,}
we have
$2m \geq 2m_{\Bb{D}} = (\Bb{D},\Gamma_{p^{\dagger}})_{z^{\dagger}}$ 
and $\gamma \geq1/2$.
Hence it suffices to show 
\begin{align}
\label{6/9,3}
2n \left( \left((n^{2}-1)\left[ \frac{m}{n}\right]-n \right)
-4\frac{(n-1)^2}{n}+\frac{n^2-1}{12n}r \right) 
+2m\left((n-1)^2  -C_{3,n}\right) \geq 0.
\end{align}
The coefficient of $m$ in (\ref{6/9,3}) 
\[
2(n^2-1)+2((n-1)^2-C_{3,n})
\]
 is non-negative.
Hence we may assume $m=n+1$ to show (\ref{6/9,3}).
Thus, it suffices to show
\begin{align*}
 &2n(n^2-1)+2(n+1)\left((n-1)^2  -C_{3,n}\right)\geq 0, \\
 &\frac{n^2-1}{12n}r -4\frac{(n-1)^2}{n}-n \geq 0.
\end{align*}
We can check it by simple calculations.
\end{proof}

\begin{prop}
\Blue{One of  $(C1)$, $(C2)$ and $(C3)$ holds.}
\end{prop}

\begin{proof}
\Red{Supposing that (C1), (C2) and (C3) do not hold,
we lead a contradiction.}
We note that $\Gamma_{p^{\dagger}} \not\subset R_{\dagger}$ in this assumption.
Hence we have $\Gamma_{p} \not \subset R$.
Take an arbitrary singular point $z \in R$.
\Blue{
Let $z^{\dagger} \in R_{\dagger}$ be the singular point such that $\Pi(z^{\dagger})=z$,
where the morphism $\Pi$ is in (\ref{diagram}).
Since (C1) does not hold,
we have $\Rm{Stab}_{T(\Ti{K}^{\dagger})_{p^{\dagger}}}(z^{\dagger} )\neq \{\Rm{id}\}$.
By Corollary~\ref{stabfree}, the singular point $z^{\dagger} \in R_{\dagger}$ is on a node of $\Gamma_{p^{\dagger}}$.
Hence the singular point $z \in R$ is on a node of $\Gamma_{p}$.
Let $\Gamma_1$ and $\Gamma_2$ be local analytic branches of $(\Gamma)_{\Rm{red}}$ at $z$.
}
Since (C2) and (C3) do not hold,
$R$ \Blue{consists of} $\Gamma_1$ and $\Gamma_2$ locally around $z$.
Hence it occurs only if $n=2$.
If $\Gamma_{p}$ is of type $l \Rm{I}_1$, we have $(\Gamma_{p})_{\Rm{red}} \subset R$.
But it is \Red{a contradiction}.
Hence we may assume that $\Gamma_{p}$ is of type $l \Rm{I}_c$ with $c \geq 2$.
Since $R$ has the singular point $z$,
we \Blue{need to} blow up at $z$ to get $\Ti{R}$.
Let $\Ov{\Gamma}_{2}$ be the irreducible component of $(\Gamma_{p})_{\Rm{red}}$
which contains $\Gamma_2$ and $\Ti{\Gamma}_{2}$ \Blue{the proper transform of} $\Gamma_2$ in  $\Ti{R}$.
By $n=2$, it holds $\Ti{\Gamma}_{2}^{2}\leq -4$.
Hence we need \Red{blowing-ups on $\Ov{\Gamma}_{2}$ at least two times} to get $\Ti{R}$.
Furthermore, 
since an arbitrary singular point of $R$ is on a node of $(\Gamma_{p})_{\Rm{red}}$,
there exists a singular point of $R$ 
on the other node of $(\Gamma_{p})_{\Rm{red}}$ consisting of $\Ov{\Gamma}_{2}$.
\Blue{
Let $z^{\dagger}_{23}$ be the node and $\Ov{\Gamma}_{3}$ another irreducible component of $(\Gamma_{p})_{\Rm{red}}$ passing through $z^{\dagger}_{23}$.
We note that the node $z^{\dagger}_{23}$ of $(\Gamma_{p})_{\Rm{red}}$ consists of the intersection of 
 $\Ov{\Gamma}_{2}$ and $\Ov{\Gamma}_{3}$.}
Since (C2) and (C3) do not hold,
$R$ consists with $\Ov{\Gamma}_{2}$ and $\Ov{\Gamma}_{3}$ locally around $z_{23}$.
Hence we have $\Ov{\Gamma}_{3} \subset R$.
Inductively, we can show $(\Gamma_{p})_{\Rm{red}} \subset R$.
But it is \Blue{a contradiction}.
\end{proof}
\subsection{The proof of Proposition~\ref{estimateK1.5}}
\Blue{Let $\Gamma_{p}$ be} of type $l\Rm{I}_{c}$.
Assume $R$ is smooth locally around $\Gamma_{p}$ and 
has a good ramification point on $\Gamma_{p}$.
We will show 
\begin{align*}
 2 n \delta  K_{f}^{2}(\Gamma_{p}) & \geq  (n-1)^{2}\sharp \Ti{K}.
 \end{align*}
 
Let $z$ be a good ramification point of $\varphi|_{R_{h}}$ and 
\Blue{$\Gamma_{p}$ the fiber through $z$ of type} $l\Rm{I}_{c}$ where $l$ and $c$ are non-negative integers.
Recall the diagram (\ref{diagram}).
\[\xymatrix{
W_{\Delta^{\dagger}} \ar[rd]  \ar@/_1.5pc/[ddr]_{\varphi_{\Delta^{\dagger}}} \ar@/^1.5pc/[rrd]^{\Pi}&  & \\
  &W_{\Delta}\times_{\Delta}\Delta^{\dagger} \ar[r] \ar[d]& W_{\Delta} \ar[d]^{\varphi_{\Delta}} \\
& \Delta^{\dagger}  \ar[r]_{\pi} &\Delta
}\]
Since the ramification index of $\varphi|_{R_{h}}:R_{h} \to B$ at $z$ is greater than $l$,
\Blue{the restriction map} $\varphi_{\Delta{\dagger}}|_{(R_{\dagger})_{h}}:(R_{\dagger})_{h} \to \Delta^{\dagger}$ is ramified 
at $z^{\dagger} \in \Pi^{-1}(z)$. 
We have 
\begin{align*}
K_{f}^{2}(\Gamma_{p^{\dagger}})\geq&
\frac{(n-1)^{2}}{n}\left(((R_{\dagger})_{h}, \Gamma_{p^{\dagger}})_{z^{\dagger}}-1\right)\sharp (T(\Ti{K}^{\dagger})_{p^{\dagger}} \cdot z^{\dagger})l .
%
%
\end{align*}
Hence we have 
\begin{align*}
K_{f}^{2}(\Gamma_{p})\geq&
\frac{(n-1)^{2}}{n}\left( ((R_{\dagger})_{h}, \Gamma_{p^{\dagger}})_{z^{\dagger}}-1\right)\sharp (T(\Ti{K}^{\dagger})_{p^{\dagger}} \cdot z^{\dagger}).
%
%
\end{align*}
So it is sufficient to show
\begin{align*}
 2 \delta (n-1)^2 \left(((R_{\dagger})_{h}, \Gamma_{p^{\dagger}})_{z^{\dagger}}-1\right) %
 - \delta (n-1)^2 \sharp( \Rm{Stab}_{T(\Ti{K}^{\dagger})_{p^{\dagger}}}(z^{\dagger}))  \geq 0 .
 \end{align*} 
From $((R_{\dagger})_{h},\Gamma_{p^{\dagger}})_{z^{\dagger}} 
\geq \sharp( \Rm{Stab}_{T(\Ti{K}^{\dagger})_{p^{\dagger}}}(z^{\dagger}))$,
it suffices to show
\begin{align*}
 2 \delta(n-1)^2 \left(((R_{\dagger})_{h}, \Gamma_{p^{\dagger}})_{z^{\dagger}}-1\right) %
 - \delta (n-1)^2 ((R_{\dagger})_{h},\Gamma_{p^{\dagger}})_{z^{\dagger}}  \geq 0 .
 \end{align*} 
\Blue{By the assumption $((R_{\dagger})_{h},\Gamma_{p^{\dagger}})_{z^{\dagger}} \geq 2$,
 we have the desired inequality}.

\section{Upper bound of the order}

Let $f:S \to B$ be a primitive cyclic covering fibration of type $(g,1,n)$
over an elliptic surface $\varphi:W \to B$.
We consider the upper bound of the order of \Red{the automorphism groups} of $f$.
Let $\Gamma_{p}$ be a fiber of $\varphi$ such that $K_{f}^{2}(\Gamma_{p})>0$.
Put $r_{p}:=\sharp \Rm{Stab}_{H}(p)$.
Since $K_{f}^{2}(\Gamma) \geq 0$ for an arbitrary fiber $\Gamma$,
we have 
\begin{align*}
K_{f}^2 \geq \sharp(H\cdot p) K_{f}^2(\Gamma_{p}).
\end{align*}
By $\sharp G=\sharp K \sharp H = n \sharp \Ti{K} \sharp H$, we have 
\begin{align*}
\sharp G \leq \frac{n\sharp \Ti{K}}{K_{f}^2(\Gamma_{p})}r_{p}K_{f}^2.
\end{align*}
Put 
\begin{align*}
\mu_{n}:=&\frac{12 n^2 \delta}{n^2 -1},\\
\mu_{n}':=&\frac{6 n^{2} \delta }{(n-1)(5n-4)}.
\end{align*}
If $K_{f}^2(\Gamma_{p}) > 0$ , we have
\[
\frac{n\sharp \Ti{K}}{K_{f}^2(\Gamma_{p})} \leq 
\mu_{n}.
\]
If $R$ has a singular point on $\Gamma_{p}$, we have 

\[
 \frac{n\sharp \Ti{K}}{K_{f}^2(\Gamma_{p})} \leq 
\mu_{n}' 
\]
from Proposition~\ref{estimateK2} and \ref{estimateK1}.
\Red{
We recall
\begin{align*}
\delta:=\Rm{min}\{ \sharp \Rm{Aut}(\Gamma_{p},O_{p})|\;
\forall p \in \Delta \textrm{ with }\Gamma_{p} \textrm{ is smooth} \}.
\end{align*}
}
We note that $\delta=2,4$ or $6$.

\begin{thm}
\label{main1}
Let $f:S \to B$ be a \Red{non-locally trivial} primitive cyclic covering fibration of type $(g,1,n)$ 
with $g \geq \Rm{max}\{\frac{30n^2 -47n +25}{n+1}, \frac{7}{2}n(n-1)+1\}$.
Put 
\begin{align*}
\mu_{n}:=&\frac{12 n^2 \delta}{n^2 -1}.
\end{align*}
Assume furthermore that when $g(B)=0$, $f$ has at least $3$ singular fibers.
Let $G$ be a finite subgroup of $\Rm{Aut} (f)$.
Then it holds
\[
  \sharp G \leq \begin{cases}
    6(2g(B)-1)\mu_{n}K_{f}^2 & (g(B) \geq 1 ),\\
     5\mu_{n}K_{f}^2 & (g(B)=0).
  \end{cases}
\]
\end{thm}

\begin{proof}
Put $r_{p}=\sharp \Rm{Stab}_{H}(p)$ for $p \in B$.
Let $\pi:B\to B/H$ be the quotient map of $B$ by $H$.
Note that $r_{p}$ \Blue{is the ramification} index of $\pi$ at $p$.

\smallskip

(i) The case of $g(B)\geq 2$.

\smallskip

We denote the genus of $B/H$ by $g(B/H)$.
From the Hurwitz formula, we get
\begin{align*}
2g(B)-2=\sharp H \left(2g(B/H)-2+\sum_{i=1}^s \frac{r_{i}-1}{r_{i}}\right), 
\end{align*}
where $s$ is the number of  ramification points and $r_{i}$ is the ramification index.
Put
\begin{align*}
T:=2g(B/H)-2+\sum_{i=1}^{s} \frac{r_{i}-1}{r_{i}},
\end{align*}
which \Blue{is positive}.
If $g(B/H)\geq 2$, then we get $\sharp H\leq g(B)-1$ by $T\geq 2$, and 
it follows that $r_{i}\leq \sharp H \leq g(B)-1$ for any $i=1, \cdots s$.

Assume that $g(B/H)=1$. Then we get  $s>0$ by $T>0$.
By $r_{i}\geq 2$ for any $i=1, \cdots s$,
we get
$1-1/r_{i}\geq 1/2.$
Therefore we obtain 
$r_{i}\leq \sharp H \leq 4(g(B)-1)$
for any $i=1, \cdots s$ by $T \geq 1/2$. 

Assume that $g(B/H)=0$.
When $s\geq 5$, we get 
$r_{i}\leq \sharp H \leq 4(g(B)-1)$
for any $i=1, \cdots s$ by $T\geq 1/2$.
When $s=4$, one of $r_{i}$ is not less than $3$.
So we get 
$r_{i}\leq \sharp H \leq 12(g(B)-1)$
for any $i=1, \cdots s$ by $T \geq 1/6$. 
When $s=3$, we may assume $r_{1}\geq r_{2}\geq r_{3}$.
By the definition of $T$,
we get  
\begin{align*}
r_{1}\leq \sharp H=\frac{2g(B)-2}{1-\frac{1}{r_{1}}-\frac{1}{r_{2}}-\frac{1}{r_{3}}},
\end{align*}
so we obtain 
\begin{align*}
r_{1}-1-\frac{r_{1}}{r_{2}}-\frac{r_{1}}{r_{3}} \leq 2g(B)-2.
\end{align*}
If $r_{3}=2$, then  $r_{2}\geq 3$, so we get 
\begin{align*}
r_{1}-1-\frac{r_{1}}{3}-\frac{r_{1}}{2} \leq 2g(B)-2.
\end{align*}
Hence we get
$r_{1}\leq 6(2g(B)-1)$.
If $r_{2}\geq r_{3}\geq 3$, we get
\Blue{$r_{1}\leq 3(2g(B)-1)$.}
Therefore we obtain 
$r_{p}\leq 6(2g(B)-1)$
for any $p\in B$.

\smallskip

(ii) The case of $g(B)=1$.

We \Blue{do not have to consider} a translation, since it has no fixed points.
Then it is well known that the order of \Red{the automorphism groups} of $B$ which fixes a point is at most $6$.
So the order of \Blue{the stabilizer} of $H$ is at most $6$, and we get $r_{p}\leq 6$.

\smallskip 

(iii) The case of $g(B)=0$

If $H$ is neither a cyclic group nor a dihedral group, the order of \Blue{the stabilizer} of $H$ is at most $5$.
So we may assume $H$ is either a cyclic group or a dihedral group.
It is well known that a rational pencil with at most two singular fibers is iso-trivial.
If the number of \Red{singular fibers} of $f$ is at least three, then 
there exist a point $p \in B$ such that $\mu_{n} \geq n \sharp \Ti{K}/K_{f}^{2}(\Gamma_{p})$
 and $r_p \leq 2$.
Hence we may assume $r_{p} \leq 2$.
\end{proof}

\begin{thm}
\label{main2}
Let $f:S \to B$ be a \Red{non-locally trivial} primitive cyclic covering fibration of type $(g,1,n)$ 
with $g \geq \Rm{max}\{\frac{30n^2 -47n +25}{n+1},  \frac{7}{2}n(n-1)+1\}$.
Put 
\begin{align*}
\mu_{n}':=\frac{6 n^{2} \delta }{(n-1)(5n-4)}.
\end{align*}
\Red{
Assume furthermore that 
the branch locus $R$ has singular points on at least three $($resp. one$)$ fibers 
when $g(B)=0$ $($resp. $g(B)\geq 1$$)$.
}
Let $G$ be a finite subgroup of $\Rm{Aut} (f)$.
Then it holds
\[
  \sharp G \leq \begin{cases}
    6(2g(B)-1)\mu_{n}'K_{f}^2 & (g(B) \geq 1 ),\\
     5\mu_{n}' K_{f}^2 & (g(B)=0).
  \end{cases}
\]
\end{thm}

\begin{cor}
\label{biell}
Let $f:S \to B$ be a \Red{non-locally trivial} bielliptic fibration 
with $g \geq 17$.
\Red{
Assume furthermore that 
the branch locus $R$ has singular points on at least three $($resp. one$)$ fibers 
when $g(B)=0$ $($resp. $g(B)\geq 1$$)$.
}
Let $G$ be a finite subgroup of $\Rm{Aut} (f)$.
Then it holds
\[
  \sharp G \leq \begin{cases}
    24\delta(2g(B)-1)K_{f}^2 & (g(B) \geq 1 ),\\
     20 \delta  K_{f}^2 & (g(B)=0).
  \end{cases}
\]
\end{cor}

\begin{cor}
\label{main3}
Let $f:S \to B$ be a non-locally trivial primitive cyclic covering fibration of type $(g,1,n)$ 
with $g \geq \Rm{max}\{\frac{30n^2 -47n +25}{n+1},  \frac{7}{2}n(n-1)+1\}$.
Put  $\Rm{Aut} (S/B):=\{(\kappa_{S},\Rm{id}_{B})\in \Rm{Aut}(f)\}$.
Assume that the branch locus $R$ has a singular point.
Then it holds
\[
  \sharp \Rm{Aut} (S/B) \leq \frac{6 n^{2} \delta }{(n-1)(5n-4)}K_{f}^2. \]
\end{cor}

\section{Example}

We construct bielliptic fibrations $f:S \to B$ with
a large automorphism group.
Let $(E, O)$ be an elliptic curve with the identity element $O \in E$.
The symbol $P \oplus Q$ denotes the sum of $P$, $Q\in E$ \Blue{as the group law}.
Furthermore, we define 
\begin{align*}
[k]P:=\underbrace{P \oplus \cdots \oplus P}_{k-summands}.
\end{align*}
Let $e$ be a positive even number and \Red{$P$ a point} of $E$ of order $e$.
We define \Red{the automorphisms} of $E$ as follows:
\begin{align*}
 \tau_{P}:E \to E \;&;\; Q  \mapsto P\oplus Q \\
 \iota:E \to E \; &; \; Q \mapsto  \ominus Q
 \end{align*}
where $\ominus Q $ denotes the inverse element of $Q$.
Let $Q_{1},Q_{2} \in E$ be distinct points such that $[2]Q_{1}=[2]Q_{2}=P$.
We note that the order of $Q_{1}$ and $Q_{2}$ is $2e$.
We define divisors
\begin{align*}
\Fr{d}_{i}&:=Q_{i}+[3]Q_{i}+\cdots +[2e-1]Q_{i}\;(i=1,2),\\
R_{E}&:=\Fr{d}_{1}+\Fr{d}_{2}.
 \end{align*}
 We recall that $P \oplus Q = R $ if and only if $P + Q \sim R +O$,
 where the symbol $\sim$ means the linearly equivalence.
Since $e$ is even, we have $\Fr{d}_{1}\sim \Fr{d}_{2} \sim eO$.
Hence we consider a double covering 
\begin{align*}
\pi:F:=
\mathrm{Spec}_{E}\left(\Ca{O}_{E}\bigoplus \Ca{O}_{E}(-\Fr{d}_{1}) \right) \to E
\end{align*}
branched along $R_{E} \in |2\Fr{d}_{1}|$.
Let $K_{E}$ be the automorphism group on $E$ generated by $\tau_{P}$ and $\iota$,
then $K_{E}$ \Blue{has order $2e$}. 
Since $\Fr{d}_{1}$ and $R_{E}$ are $K_{E}$-stable, there exists $\Ti{\kappa} \in \Rm{Aut}(F)$
for any $\kappa \in K_{E}$ such that the diagram 
\begin{equation*}
\xymatrix{
F  \ar[r]^{\Ti{\kappa}} \ar[d]_{\pi} & F \ar[d]^{\pi} \\
E \ar[r]_{\kappa}           & E
}
\end{equation*}
commutes.

Put $W:=\Bb{P}^{1} \times E$ and let $\varphi:W \to \Bb{P}^1$ be the projection.
An action of $\kappa \in K_{E}$ induces the following \Blue{action on} $W$:
\[
\kappa: W \to W\; : \; (x,y)\mapsto (x,\kappa(y))
\]
Thus, we can regard $K_{E}$ as a subgroup of $\Rm{Aut}(W/\Bb{P}^1)$.
Let $R_{h}$ and $\Fr{d}_{h}$ be \Blue{the pullback} of $R_{E}$ and $\Fr{d}_{1}$ 
by the natural projection $W \to E$.
%

We recall that the icosahedral group $I_{60}$ acts $\Bb{P}^{1}$.
We denote by $p_{1},\cdots, p_{12}$ the $I_{60}$-orbit on $\Bb{P}^{1}$ with length $12$.
Let $\Gamma_{i}$ be the fiber of $\varphi$ over $p_{i}$.
We define the action on $W$ for $h \in I_{60}$ as follows.
\[
h: W \to W\;;\; (x,y)\mapsto (h(x),y)
\]
Hence we define an automorphism group 
\[
\Ti{G}:=I_{60}\times K_{E} \cong I_{60} K_{E} \subset \Rm{Aut}(\varphi).
\]
The order of $\Ti{G}$ is $120e$.
We define divisors
\[
\Fr{d}:=\Fr{d}_{h}+\sum_{i=1}^{6}\Gamma_{i},\quad
R:=R_{h}+\sum_{i=1}^{12}\Gamma_{i}.
\]
Since $R_{E}$ and $2\Fr{d}_{1}$ are \Blue{linearly equivalent},
we have $R \in |2\Fr{d}|$.
Hence \Blue{we can take} a double covering 
\begin{align*}
\theta:S_{0}:=
\mathrm{Spec}_{W}\left(\Ca{O}_{W}\bigoplus \Ca{O}_{W}(-\Fr{d}) \right) \to W
\end{align*}
branched along $R \in |2\Fr{d}|$.
Let $S$ be the minimal resolution of $S_{0}$.
Then the induced fibration $f:S \to \Bb{P}^{1}$ is a bielliptic fibration with $g=l+1$.
By the assumption of $\Ti{G}$,
we have $\kappa^{\ast}\Fr{d}\sim\Fr{d}$ and $\kappa^{\ast}R=R$.
Thus, any action $\kappa$ induces an action $\Ti{\kappa} \in \Rm{Aut}(f)$.
Hence we can construct a subgroup $G$ of $\Rm{Aut} (f)$ generated by $\Ti{\kappa}$'s and 
the covering transformation.
Then $G$ is a finite subgroup with $\sharp G \geq 240e$.

On the other hand, 
we have
\[
K_{f}^{2}(\Gamma_{i})=2e.
\]
since $R$ has $e$ ordinary double points on each $\Gamma_{i}$.
Thus, we have $K_{f}^{2}=24e$ and it holds that $\sharp G \geq 10K_{f}^{2}$.

\vspace{2\baselineskip}
{}

\bigskip
\bigskip

Hiroto Akaike

Department of Mathematics, 
Graduate School of Science, 
Osaka University,

1-1 Machikaneyama, Toyonaka, Osaka 
560-0043, Japan

e-mail: u802629d@ecs.osaka-u.ac.jp


\begin{thebibliography}{9}
\bibitem{Aka} H.\ Akaike, Bounds for the order of automorphism groups of cyclic covering fibrations of a ruled surface,
arxiv:2011.03000v3[math.AG].
\bibitem{Ara} T.\ Arakawa, Bounds for the order of automorphism groups of hyperelliptic fibrations, T\^ohoku.\ Math.\ J. \textbf{50} (1998), 317--323.
\bibitem{Bar} M.\ A.\ Barja, On the slope of bielliptic fibrations, Proc.\ Amer.\ Math.\ Soc.\ \textbf{129} (2001), 1899--1906.  
 \bibitem{Chen} Z.\ Chen, Best bounds of automorphism groups of hyperelliptic fibrations, T\^ohoku.\ Math.\ J. \textbf{50} (1998), 469--489.
 \bibitem{Eno} M.\ Enokizono, Slopes of fibered surfaces with a finite cyclic automorphism, 
 Michigan Math.\ J.\ \textbf{66} (2017), 125--154.
 \bibitem{Eno2} M.\ Enokizono, Upper bounds on the slope of certain fibered surfaces, 
 J.\ Math.\ Soc.\ Japan. Advance Publication (2022), 1--51, https://doi.org/10.2969/jmsj/86278627
\bibitem{Kod} K.\ Kodaira, On compact analytic surfaces I, Ann.\ of Math. \textbf{71} (1960),
111--152; II, ibid.\ \textbf{77} (1963), 563--626;
III, ibid.\ \textbf{78} (1963), 1--40.
\bibitem{Kon} K.\ Konno, A lower bound of the slope of trigonal fibrations, Intern.\ J.\ Math. \textbf{7} (1996), 19--27.
\bibitem{Mir} R.\ Miranda and U.\ Persson, Torsion groups of elliptic surfaces, Compositio Math \textbf{72} (1989), 249--267.
 \bibitem{Xiao1} G.\ Xiao, Bound of automorphism of surfaces of general type, I, Ann.\ of Math. \textbf{139} (1994), 51--77.
 \bibitem{Xiao2} G.\ Xiao, Bound of automorphism of surfaces of general type, II, J.Algebraic Geom. \textbf{4} (1995), 701--793.
\end{thebibliography}
\end{document}